\documentclass[10pt,reqno]{article}

\usepackage[b5paper,top=1.2in,left=0.9in]{geometry}
\usepackage{verbatim}
\usepackage{amssymb}
\usepackage{amsmath}
\usepackage{graphicx}
\usepackage{appendix}
\usepackage{color}
\usepackage{amsthm}
\usepackage{tikz}
\usepackage{float}
\usepackage{subcaption}
\usetikzlibrary{arrows,positioning,decorations.pathmorphing,  decorations.markings}

\renewcommand{\d}{\mathrm{d}}
\newcommand{\D}{\mathrm{D}}
\newcommand{\e}{\mathrm{e}}

\newtheorem{Thm}{Theorem}[section]
\newtheorem{Lem}[Thm]{Lemma}
\newtheorem{Prop}[Thm]{Proposition}
\newtheorem{Cor}[Thm]{Corollary}
\newtheorem{Rem}[Thm]{Remark}
\newtheorem{Def}[Thm]{Definition}
\newtheorem{Con}[Thm]{Conjecture}
\newtheorem{Ex}[Thm]{Example}

\newtheorem*{MainThm}{Main Theorem}

\newtheoremstyle{named}{}{}{\itshape}{}{\bfseries}{.}{.5em}{#1 #3}
\theoremstyle{named}

\def\R{\mathbb{R}}
\def\Q{\mathbb{Q}}
\def\N{\mathbb{N}}
\def\C{\mathbb{C}}
\def\Z{\mathbb{Z}}

\def\fn{\mathfrak{n}}

\def\fb{\mathfrak{b}}
\def\g{\mathfrak{g}}
\def\fh{\mathfrak{h}}

\def\sl{\mathfrak{sl}}

\def\cI{\mathcal{I}}

\def\cP{\mathcal{P}}

\def\cR{\mathcal{R}}
\def\cS{\mathcal{S}}

\def\cU{\mathcal{U}}
\def\cV{\mathcal{V}}
\def\cW{\mathcal{W}}

\def\a{\alpha}
\def\b{\beta}

\def\D{\Delta}

\def\d{\delta}
\def\e{\epsilon}

\def\l{\lambda}
\def\L{\Lambda}

\def\s{\sigma}

\def\w{\omega}

\def\bC{\textbf{C}}

\def\be{\textbf{e}}

\def\bf{\textbf{f}}

\def\bi{\textbf{i}}

\def\bo{\textbf{o}}

\def\bQ{\textbf{Q}}

\def\=>{\Longrightarrow}

\def\iff{\Longleftrightarrow}

\def\to{\longrightarrow}

\def\ox{\otimes}
\def\o+{\oplus}
\def\bo+{\bigoplus}
\def\x{\times}
\def\<{\langle}
\def\>{\rangle}
\def\({\left(}
\def\){\right)}
\def\oo{\infty}

\def\^{\wedge}
\def\+{\dagger}

\def\inv{^{-1}}
\def\half{\frac{1}{2}}

\def\dd[#1,#2]{\frac{d#1}{d#2}}
\def\del[#1,#2]{\frac{\partial #1}{\partial #2}}
\def\over[#1]{\overline{#1}}
\def\vec[#1]{\overrightarrow{#1}}

\def\tab{\;\;\;\;\;\;}

\newcommand{\til}[1]{\widetilde{#1}}
\newcommand{\what}[1]{\widehat{#1}}

\newcommand{\case}[2][cccccccccccccccccccccccccccccccccccccccccc]{\left\{\begin{array}{#1}#2 \\ \end{array}\right.}
\newcommand{\Eq}[1]{\begin{align}#1\end{align}}
\newcommand{\Eqn}[1]{\begin{align*}#1\end{align*}}

\begin{document}
\title{Positive Casimir and Central Characters of \\Split Real Quantum Groups}

\author{  Ivan C.H. Ip\footnote{
         	Kavli Institute for the Physics and Mathematics of the Universe (WPI), 
		The University of Tokyo, 
		Kashiwa, Chiba 
		277-8583, Japan
		\newline
		Email: ivan.ip@ipmu.jp
          }
}

\date{\today}

\numberwithin{equation}{section}

\maketitle

\textit{Dedicated to the memory of my grandfather N. Y. Wong (1926-2015)}

\bigskip

\begin{abstract}
We describe the generalized Casimir operators and their actions on the positive representations $\cP_\l$ of the modular double of split real quantum groups $\cU_{q\til{q}}(\g_\R)$. We introduce the notion of virtual highest and lowest weights, and show that the central characters admit positive values for all parameters $\l$. We show that their image defines a semi-algebraic region bounded by real points of the discriminant variety independent of $q$, and we discuss explicit examples in the lower rank cases.
\end{abstract}

{\small {\textbf{Keywords.} Modular double, quantum group, Casimir operator, central character, discriminant variety}


\newpage
\tableofcontents
\section{Introduction}\label{sec:intro}
The notion of the \emph{positive principal series representations}, or simply \emph{positive representations}, was introduced in \cite{FI} as a new research program devoted to the representation theory of split real quantum groups $\cU_{q\til{q}}(\g_\R)$. It uses the concept of modular double for quantum groups \cite{Fa1, Fa2}, and has been studied for $\cU_{q\til{q}}(\sl(2,\R))$ by Teschner \textit{et al.} \cite{BT, PT1, PT2}. Explicit construction of the positive representations $\cP_\l$ of  $\cU_{q\til{q}}(\g_\R)$ associated to a simple Lie algebra $\g$ has been obtained for the simply-laced case in \cite{Ip2} and non-simply-laced case in \cite{Ip3}, where the generators of the quantum groups are realized by positive essentially self-adjoint operators. Furthermore, since the generators are represented by positive operators, we can take real powers by means of functional calculus, and we obtained the so-called \emph{transcendental relations} of the (rescaled) generators (cf. \eqref{smallef}):
\Eq{\til{\be_i}=\be_i^{\frac{1}{b_i^2}},\tab \til{\bf_i}=\bf_i^{\frac{1}{b_i^2}}, \tab \til{K_i}=K_i^{\frac{1}{b_i^2}},\tab i=1,..., n}
where $n$=rank $\g$, giving the self-duality between different parts of the modular double, while in the non-simply-laced case, new explicit analytic relations between the quantum group and its Langlands dual have been observed \cite{Ip3}.

One important open problem is the study of the tensor product decompositions of the positive representation $\cP_\l\ox \cP_\mu$. It is believed that the positive representations are closed under taking tensor product, which, together with the existence of the universal $R$ operator \cite{Ip4}, lead to the construction of new classes of \emph{braided tensor category} and hence to further applications parallel to those from the representation theory of compact quantum groups. In a recent work \cite{Ip6}, we studied the tensor product decomposition restricted to the positive Borel part $\cU_{q\til{q}}(\fb_\R)$ which is defined in the $C^*$-algebraic setting by certain multiplier Hopf-* algebra, and we show that the decomposition is closely related to the so-called quantum higher Teichm\"{u}ller theory \cite{FG1,FG2}, and provides evidences for the decomposition of the full quantum group in general. However, taking the full quantum group into account is much more difficult. In the case of $\cU_{q\til{q}}(\sl(2,\R))$, this was accomplished in \cite{NT, PT2} by means of the decomposition of the Casimir operator
\Eq{\bC = FE+\frac{qK+q\inv K\inv}{(q-q\inv)^2}}
which is central in $\cU_{q\til{q}}(\sl(2,\R))$ and its spectral decomposition in $\cP_\l\ox\cP_\mu$ uniquely determines its decomposition into positive representations. Therefore we are interested in the higher rank situation by looking at the appropriate central elements and study its spectral decomposition. 

One important object to study in any representation theory is the notion of \emph{central characters}, i.e. the action of the center of the algebra on irreducible representations, which necessarily act as scalars. It is well-known that the center of $\cU_q(\g)$ is generated by rank $\g$ elements which we called the generalized Casimir operators. In this paper, we will take the \emph{same} set of generators of the center of the compact $\cU_q(\g)$ as our generalized Casimir operators for the split real quantum group $\cU_{q\til{q}}(\g_\R)$. We calculate the eigenvalues of the generalized Casimir operators acting on the positive representation $\cP_\l$, which is irreducible. Their actions are well-defined on the core of $\cP_\l$, which extends to the domain of definition of the unbounded operators represented by the generators $\{\be_i, \bf_i, K_i\}$ of the split real quantum group.

There are many ways to choose the generating set of the center as the generalized Casimir operators \cite{Et, RTF, ZGB}. In this paper, we will use a modified version from \cite{KS} and define 
\Eq{\bC_k = Tr|_{V_k}^q(RR_{21}),\tab k=1,...,n}
where $V_k$ is the fundamental representations of the \emph{compact} quantum group $\cU_q(\g)$, $R$ is the universal $R$ matrix, and $Tr^q$ is the quantum trace taken over the representation $V_i$ of $\cU_q(\g)$. However, in our construction, $Tr^q$ is twisted instead by a central element $u$ that is related to the antipode of the \emph{modular double} $\cU_{q\til{q}}(\g_\R)$. These operators $\bC_k$ are then acted on the positive representations $\cP_\l$ as certain scalars. Our main results is the following (Theorem \ref{main})

\begin{MainThm} The operators $\bC_k$ acts on $\cP_\l$ as scalars $C_k(\l)$ which are positive, and bounded below by $\dim V_k$ for every parameters $\l=(\l_1,...,\l_n)\in\R_{\geq 0}^n$. Furthermore, they can be given by the Weyl character formula.
\end{MainThm}

We also note that unlike the case of compact quantum groups, by rescaling the parameters $\l_i$, the eigenvalues $C_k(\l)$ are \emph{independent} on $q$. Together with the conjecture that the positive representations should be closed under tensor products \cite{Ip5}, this implies

\begin{Con} The coproduct $\D(\bC_k)$ of the generalized Casimir operators $\bC_k$ acting on $\cP_\l\ox\cP_\mu$ is a positive operator, with spectrum bounded below by $\dim V_k$ for all $k=1,..,n$, and they can be simultaneously diagonalized.
\end{Con}

Therefore we will call these generalized Casimir operators the \emph{positive Casimirs} of the modular double of split real quantum groups $\cU_{q\til{q}}(\g_\R)$. In particular, we believe that just like the case in $\cU_{q\til{q}}(\sl(2,\R))$, the decomposition of these operators in the tensor product will give explicitly the decomposition of $\cP_\l\ox\cP_\mu$.

In the classical case, one calculate the eigenvalues of $\bC_k$ by acting on the highest weight vector, which kills the non-Cartan part, giving the (quantum) Harish-Chandra homomorphism projecting onto the Cartan part, where the eigenvalues can be calculated by the information from the highest weight. In the case of positive representation, the situation is fundamentally different because the representation $\cP_\l$ is infinite dimensional, and we do not have a highest or lowest weight vector. In this paper, we introduced the notion of \emph{virtual highest/lowest weight} to deal with this problem, which follows from a new combinatorial description of the positive roots (Proposition \ref{root-com})
\Eq{
s_{i_1}s_{i_2}...s_{i_t}(\a_k) = \a_k-\sum_{j=1}^t a_{i_jk}s_{i_1}s_{i_2}...s_{i_{j-1}}(\a_{i_j}),
}
where $s_i$ are the root reflections, $\a_k$ are the positive simple roots, and $a_{ij}$ are the Cartan matrix elements.

We observed two aspects that might relate the results of the split real case to that of the compact case. Using the virtual highest and lowest weights, we found that each component of these weights look like certain analytic continuation to the complex line $\frac{Q}{2}+\bi \R_{>0}^+$. On the other hand, we also note that the calculations of the action of the generalized Casimir operators depend solely on the information provided by the finite dimensional fundamental representations of the compact quantum groups. This interplay between (finite dimensional) representations $V_k$ of the compact group, and its action on the (infinite dimension) positive representations $\cP_\l$ of the split real quantum group is interesting on its own. We believe that these concepts can be carried further to provide more analogous results from the compact quantum groups to their split real version, including the tensor product decomposition and categorification, by means of certain analytic continuations. These will be explored further in separate publications.

Finally, to study the spectral decomposition of the tensor product of Casimir operators, one is interested in the region $\cR$ defined by the eigenvalues of these operators. This is the region defined by the image of
 \Eq{\Phi:\R_{\geq 0}^n\simeq &\to \R_{\geq 0}^n\\
(\l_1,...,\l_n) &\mapsto (C_1(\l),...,C_n(\l))\nonumber,
}
and it turns out that the image is a semi-algebraic set bounded by the so-called \emph{discriminant variety} (or \emph{divisor}) \cite{Gr, Sai1, Sai2}, which is studied extensively from a totally different point of view in the theory of \emph{primitive forms}. This provides a homeomorphism between the positive Weyl chamber with the image $\cR$, and the origin is mapped to a cusp-like singularity, given at $(d_1,...,d_n)$ where $d_k$ is the dimension of the fundamental representations $V_k$. Furthermore, these images are again \emph{independent} on the choice of the quantum parameter $q$. We will provide the descriptions of these regions in the rank 2 and 3 cases, together with some explicit expressions of the generalized Casimir operators themselves.

The paper is organized as follows. In Section \ref{sec:prelim}, we fixed the notation for roots systems and the Drinfeld-Jimbo quantum groups, and recall the definition of positive representations of split real quantum groups as well as the construction of the universal $R$ matrix. In Section \ref{sec:sl2}, we recall the situation in $\cU_{q\til{q}}(sl(2,\R))$ due to \cite{BT, PT2}. In Section \ref{sec:genCasimir} we discuss the construction of the generalized Casimir operators. In Section \ref{sec:virtual}, we introduce the notion of virtual highest and lowest weight in order to calculate the central characters of the Casimirs in Section \ref{sec:cc}. Finally in Section \ref{sec:disc} we describe the regions defined by the central characters of the generalized Casimir operators, and discuss in Section \ref{sec:lowrank} some explicit examples in the low rank cases. In the appendix we recall the dimensions of the fundamental representations $V_k$ and also provide the graphs of $\cR$ in lower rank cases.

\textbf{Acknowledgment.} I would like to thank Kyoji Saito and Todor Milanov for stimulating discussions about the discriminant varieties and Weyl invariant polynomials. This work was supported by World Premier International Research Center Initiative (WPI Initiative), MEXT, Japan.
\section{Preliminaries}\label{sec:prelim}
Throughout the paper, we will fix once and for all $q=e^{\pi \bi b^2}$ with $\bi =\sqrt{-1}$, ${0<b^2<1}$ and $b^2\in\R\setminus\Q$. We also denote by $Q=b+b\inv$. Also let $I=\{1,2,...,n\}$ denotes the set of nodes of the Dynkin diagram of the simple Lie algebra $\g$ where $n=rank(\g)$.
\subsection{Notation for roots and weights}\label{sec:prelim:roots}
\begin{Def} Let $(-,-)$ be the $W$-invariant inner product of the root lattice where $W$ is the Weyl group of the Cartan datum. Let $\a_i\in\D^+$, $i\in I$ be the positive simple roots, and we define
\begin{eqnarray}
a_{ij}:=\frac{2(\a_i,\a_j)}{(\a_i,\a_i)},
\end{eqnarray}
where $A=(a_{ij})$ is the Cartan matrix. 
\end{Def}
\begin{Def}\label{rho}
We denote by $H_i\in\fh$ the coroot corresponding to the positive simple root $\a_i$. The fundamental weights $w_i\in\fh_\R^*$ are dual to the simple coroots and are given by
\Eq{w_i := \sum (A\inv)_{ji} \a_j,}
such that $(w_i, \a_j^\vee)=\d_{ij}$, where the coroot $H_j:=\a_j^\vee := \frac{2\a_j}{(\a_j,\a_j)}$ in the root lattice. 

Similarly, we denote the corresponding fundamental coweights $W_i\in\fh_\R$ in the real span of $\fh$ dual to the simple roots by 
\Eq{W_i := \sum (A\inv)_{ij}H_j.}
\end{Def}
\begin{Def}
We let 
\Eq{\rho := \frac{1}{2}\sum_{\a\in\D^+}\a = \sum_i w_i=\sum_i d_i W_i} be the half sum of positive roots, or equivalently, the sum of fundamental weights, or the rescaled sum of fundamental coweights, where $d_i = \frac{(\a_i,\a_i)}{2}$.
\end{Def}

\begin{Def} For $(\l_1,...,\l_n)\in\R^n$ and $b_j\in\R$, let 
\Eq{\vec[\l_b]:=\sum_{j=1}^n \l_jb_jW_j\in \fh_\R\label{lb}}
be a vector in the real span of $\fh$.
\end{Def}
\begin{Prop} The Weyl group action on the positive simple roots are given by 
\Eq{s_i\cdot \a_j = \a_j - a_{ij}\a_i,}
where $s_i\in W$ is the simple reflection corresponding to the root $\a_i$.
Then 
\Eq{s_i\cdot W_j = W_j - \d_{ij}\a_j^\vee = W_j - \d_{ij}\sum_{k=1}^na_{jk}W_k,} and the Weyl group action of $\vec[\l_b]$ is given by
\Eq{s_i\cdot \vec[\l_b] &= \sum_{j=1}^n \l_jb_jW_j-\l_ib_i\a_i^\vee\\
&=\sum_{j=1}^n (\l_jb_j-a_{ij}\l_ib_i)W_j.\label{lb-weyl}
}
\end{Prop}

\subsection{Definition of $\cU_q(\g)$ and $\cU_{q\til{q}}(\g_\R)$}\label{sec:prelim:Uqgr}
In order to fix the convention we use throughout the paper, we recall the definition of the Drinfeld-Jimbo quantum group $\cU_q(\g_\R)$ where $\g$ is a simple Lie algebra of general type \cite{D, J}.
\begin{Def} \label{qi} Let $\a_i$, $i\in I$ be the positive simple roots, and we define
\begin{eqnarray}
q_i:=q^{\frac{1}{2}(\a_i,\a_i)}:=e^{\pi \bi  b_i^2},
\end{eqnarray}
We will let $\a_1$ be the short root in type $B_n$ and the long root in type $C_n, F_4$ and $G_2$.

We choose 
\Eq{\frac{1}{2}(\a_i,\a_i):=\case{1&\mbox{$i$ is long root or in the simply-laced case,}\\\frac{1}{2}&\mbox{$i$ is short root in type $B,C,F$,}\\\frac{1}{3}&\mbox{$i$ is short root in type $G_2$,}}}and $(\a_i,\a_j)=-1$ when $i,j$ are adjacent in the Dynkin diagram.

Therefore in the case when $\g$ is of type $B_n$, $C_n$ and $F_4$, if we define $b_l=b$, and $b_s=\frac{b}{\sqrt{2}}$ we have the following normalization:
\begin{eqnarray}
\label{qiBCF}q_i=\case{e^{\pi \bi b_l^2}=q&\mbox{$i$ is long root,}\\e^{\pi \bi b_s^2} =q^{\frac{1}{2}}&\mbox{$i$ is short root.}}
\end{eqnarray}
In the case when $\g$ is of type $G_2$, we define $b_l=b$, and $b_s=\frac{b}{\sqrt{3}}$, and we have the following normalization:
\begin{eqnarray}
\label{qiG}q_i=\case{e^{\pi \bi b_l^2}=q&\mbox{$i$ is long root,}\\e^{\pi \bi b_s^2} =q^{\frac{1}{3}}&\mbox{$i$ is short root.}}
\end{eqnarray}
\end{Def}

\begin{Def} Let $A=(a_{ij})$ denotes the Cartan matrix. Then $\cU_q(\g)$ with $q=e^{\pi\bi b_l^2}$ is the algebra generated by $E_i$, $F_i$ and $K_i^{\pm1}$, $i\in I$ subject to the following relations:
\begin{eqnarray}
K_iE_j&=&q_i^{a_{ij}}E_jK_i,\\
K_iF_j&=&q_i^{-a_{ij}}F_jK_i,\\
{[E_i,F_j]} &=& \d_{ij}\frac{K_i-K_i\inv}{q_i-q_i\inv},
\end{eqnarray}
together with the Serre relations for $i\neq j$:
\begin{eqnarray}
\sum_{k=0}^{1-a_{ij}}(-1)^k\frac{[1-a_{ij}]_{q_i}!}{[1-a_{ij}-k]_{q_i}![k]_{q_i}!}E_i^{k}E_jE_i^{1-a_{ij}-k}&=&0,\label{SerreE}\\
\sum_{k=0}^{1-a_{ij}}(-1)^k\frac{[1-a_{ij}]_{q_i}!}{[1-a_{ij}-k]_{q_i}![k]_{q_i}!}F_i^{k}F_jF_i^{1-a_{ij}-k}&=&0,\label{SerreF}
\end{eqnarray}
where 
\Eq{[k]_q:=\frac{q^k-q^{-k}}{q-q\inv}.}
\end{Def}

\begin{Def}\label{abuse}
By abuse of notation, we will denote formally
\Eq{K_i =: q_i^{H_i}} where $H_i\in \fh$ is the simple coroot in the Cartan subalgebra. Furthermore, we allow fractional powers of $K_i$ by adjoining the elements $K_i^\frac{1}{c}$ into $\cU_q(\g)$, where $c=\det A\in\N$ is the determinant of the Cartan matrix, and again denoting the resulting algebra by $\cU_q(\g)$.
\end{Def}

We choose the Hopf algebra structure of $\cU_q(\g)$ to be given by
\Eq{
\D(E_i)=&1\ox E_i+E_i\ox K_i,\\
\D(F_i)=&K_i^{-1}\ox F_i+F_i\ox 1,\\
\D(K_i)=&K_i\ox K_i,\\
\e(E_i)=&\e(F_i)=0,\tab \e(K_i)=1,\\
S(E_i)=&-q_iE_i, \tab S(F_i)=-q_i\inv F_i,\tab S(K_i)=K_i\inv.
}

We define $\cU_q(\g_\R)$ to be the real form of $\cU_q(\g)$ induced by the star structure
\Eq{E_i^*=E_i,\tab F_i^*=F_i,\tab K_i^*=K_i.}
Finally, according to the results of \cite{Ip2,Ip3}, we define the modular double $\cU_{q\til{q}}(\g_\R)$ to be
\Eq{\cU_{q\til{q}}(\g_\R):=\cU_q(\g_\R)\ox \cU_{\til{q}}(\g_\R),&\tab \mbox{$\g$ is simply-laced,}\\
\cU_{q\til{q}}(\g_\R):=\cU_q(\g_\R)\ox \cU_{\til{q}}({}^L\g_\R),&\tab \mbox{otherwise,}}
where $\til{q}=e^{\pi \bi b_s^{-2}}$, and ${}^L\g_\R$ is the Langlands dual of $\g_\R$ obtained by interchanging the long roots and short roots of $\g_\R$.

\subsection{Positive representations of $\cU_{q\til{q}}(\g_\R)$}\label{sec:prelim:pos}
In \cite{FI, Ip2,Ip3}, a special class of representations for $\cU_{q\til{q}}(\g_\R)$, called the positive representations, is defined. The generators of the quantum groups are realized by positive essentially self-adjoint operators, and also satisfy the so-called \emph{transcendental relations}, relating the quantum group with its modular double counterpart. More precisely, we have
\begin{Thm} Let the rescaled generators be
\Eq{\be_i:=\left(\frac{\bi}{q_i-q_i\inv}\right)\inv E_i,\tab \bf_i:=\left(\frac{\bi}{q_i-q_i\inv}\right)\inv F_i.\label{smallef}}
Note that $$\left(\frac{\bi}{q_i-q_i\inv}\right)\inv=2\sin(\pi b_i^2)>0.$$
Then there exists a family of representations $\cP_{\l}$ of $\cU_{q\til{q}}(\g_\R)$ parametrized by the $\R_+$-span of the cone of positive weights $\l\in P_\R^+\subset \fh_\R^*$, or equivalently by $\l:=(\l_1,...,\l_n)\in \R_+^n$ where $n=rank(\g)$, such that 
\begin{itemize}
\item The generators $\be_i,\bf_i,K_i$ are represented by positive essentially self-adjoint operators acting on $L^2(\R^{l(w_0)})$, where $l(w_0)$ is the length of the longest element $w_0\in W$ of the Weyl group.
\item Define the transcendental generators:
\Eq{\til{\be_i}:=\be_i^{\frac{1}{b_i^2}},\tab \til{\bf_i}:=\bf_i^{\frac{1}{b_i^2}},\tab \til{K_i}:=K_i^{\frac{1}{b_i^2}}.\label{transdef}}
Then \begin{itemize}
\item if $\g$ is simply-laced, the generators $\til{\be_i},\til{\bf_i},\til{K_i}$ are obtained by replacing $b$ with $b\inv$ in the representations of the generators $\be_i,\bf_i,K_i$. \\
\item If $\g$ is of type $B,C,F,G$, then the generators $\til{E_i},\til{F_i},\til{K_i}$ with
\Eq{\til{\be_i}:=\left(\frac{\bi}{\til{q_i}-\til{q_i}\inv}\right)\inv \til{E_i},\tab \til{\bf_i}:=\left(\frac{\bi}{\til{q_i}-\til{q_i}\inv}\right)\inv \til{F_i}}
generate $\cU_{\til{q}}({}^L\g_\R)$ defined in the previous section. Here $\til{q_i}=e^{\pi ib_i^{-2}}$.
\end{itemize}
\item The generators $\be_i,\bf_i,K_i$ and $\til{\be_i},\til{\bf_i},\til{K_i}$ commute weakly up to a sign.
\end{itemize}
\end{Thm}

The positive representations are constructed for each reduced expression $w_0\in W$ of the longest element of the Weyl group, and representations corresponding to different reduced expressions are unitary equivalent.
\begin{Def}\label{variables} Fix a reduced expression of $w_0=s_{i_1}...s_{i_N}$. Let the coordinates of $L^2(\R^N)$ be denoted by $\{u_i^k\}$ so that $i$ is the corresponding root index, and $k$ denotes the sequence this root is appearing in $w_0$ from the right. Also denote by $\{v_j\}_{j=1}^N$ the same set of coordinates counting from the left, and $v(i,k)$ the index such that $u_i^k=v_{v(i,k)}$.
\end{Def}
\begin{Ex} The coordinates of $L^2(\R^6)$ for $A_3$ corresponding to $w_0=s_3s_2s_1s_3s_2s_3$ is given by
$$(u_3^3,u_2^2,u_1^1,u_3^2, u_2^1,u_3^1)=(v_1,v_2,v_3,v_4,v_5,v_6).$$
\end{Ex} 
\begin{Def}We denote by $p_u=\frac{1}{2\pi \bi}\del[,u]$ and 
\Eq{e(u)&:=e^{\pi bu},\tab [u]:=q^\half e(u)+q^{-\half}e(-u),}
so that
\Eq{\label{standardform}[u]e(-2p):=(q^\half e^{\pi bu}+q^{-\half}e^{-\pi bu})e^{-2\pi bp} = e^{\pi b(u-2p)}+e^{\pi b(-u-2p)}}
is positive whenever $[p,u]=\frac{1}{2\pi \bi}$.
\end{Def}

\begin{Def}\label{usul}By abuse of notation, we denote by
\Eq{[u_s+u_l]e(-2p_s-2p_l):=e^{\pi b_s(-u_s-2p_s)+\pi b_l(-u_l-2p_l)}+e^{\pi b_s(u_s-2p_s)+\pi b_l(u_l-2p_l)},}
where $u_s$ (resp. $u_l$) is a linear combination of the variables corresponding to short roots (resp. long roots). The parameters $\l_i$ are also considered in both cases. Similarly $p_s$ (resp. $p_l$) are linear combinations of the $p$ shifting of the short roots (resp. long roots) variables. This applies to all simple $\g$, with the convention given in Definition \ref{qi}.
\end{Def}

\begin{Thm}\label{FKaction}\cite{Ip2,Ip3} Using the notation of Definition \ref{usul}, for a fixed reduced expression of $w_0=s_{i_1}...s_{i_N}$, the positive representation $\cP_\l$ is given by
\Eq{
\bf_i=&\sum_{k=1}^m\left[-\sum_{j=1}^{v(i,k)-1} a_{i_j,i}v_j-u_i^k-2\l_i\right]e(2p_{i}^{k}),\label{FF}\\
K_i=&e\left(-\sum_{k=1}^{l(w_0)} a_{i_k,i}v_k-2\l_i\right).\label{KK}
}
where $m$ is the number of root index $i$ appearing in $w_0$.
By taking $w_0=w's_i$ so that the simple reflection for root $i$ appears on the right, the action of $\be_i$ is given by
\Eq{\label{EE} 
\be_i=&[u_i^1]e(-2p_i^1).}
\end{Thm}

Let us recall the explicit expression for the positive representations in the case of $\cU_{q\til{q}}(\sl(2,\R))$. For details of the construction and the other cases please refer to \cite{Ip2,Ip3}. 

\begin{Prop}\cite{BT,PT2}\label{canonicalsl2} The positive representation $\cP_\l$ of $\cU_{q\til{q}}(\sl(2,\R))$ acting on $L^2(\R$) by positive unbounded essentially self-adjoint operators is given by
\Eqn{
\be=&[u-\l]e(-2p):=e^{\pi b(-u+\l-2p)}+e^{\pi b(u-\l-2p)},\\
\bf=&[-u-\l]e(2p):=e^{\pi b(u+\l+2p)}+e^{\pi b(-u-\l+2p)},\\
K=&e(-2u):=e^{-2\pi bu}.
}
Note that it is unitary equivalent to the canonical form \eqref{FF}-\eqref{EE} by $u\mapsto u+\l$.
\end{Prop}
\subsection{Universal $R$ matrices for $\cU_{q}(\g)$}\label{sec:prelim:Rmatrix}
Here we summarize the construction of universal $R$ matrices for the braiding of $\cU_q(\g)$. Let $q:=e^{h/2}$. It is known \cite{D,J} that for the quantum group $\cU_h(\g)$ as a $\C[[h]]$-algebra completed in the $h$-adic topology, one can associate certain canonical, invertible element $R$ in an appropriate completion of $(\cU_h(\g))^{\ox 2}$ such that the the braiding relation
\Eq{\D'(X)R:=(\s \circ \D)(X)R=R\D(X), \tab \s(x\ox y)=y\ox x\label{br}} 
is satisfied.

For the quantum group $\cU_h(\g)$ associated to the simple Lie algebra $\g$, an explicit multiplicative formula has been computed independently in \cite{KR} and \cite{LS}, where the central ingredient involves the quantum Weyl group which induces Lusztig's isomorphism $T_i$. Explicitly, let 
\Eq{[U,V]_q:=qUV-q\inv VU} be the $q$-commutator.
\begin{Def}\cite{KR, Lu1}\label{Lus} The Lusztig's isomorphism is given by
\Eq{T_i(K_j)=K_jK_i^{-a_{ij}}, \tab T_i(E_i)=-q_iF_iK_i\inv, \tab T_i(F_i)=-q_i\inv K_iE_i,}
\Eq{T_i(E_j)=&(-1)^{a_{ij}}\frac{1}{[-a_{ij}]_{q_i}!}\left[\left[E_i,...[E_i,E_j]_{q_i^{\frac{a_{ij}}{2}}}\right]_{q_i^{\frac{a_{ij}+2}{2}}}...\right]_{q_i^{\frac{-a_{ij}-2}{2}}},\\
T_i(F_j)=&\frac{1}{[-a_{ij}]_{q_i}!}\left[\left[F_i,...[F_i,F_j]_{q_i^{\frac{a_{ij}}{2}}}\right]_{q_i^{\frac{a_{ij}+2}{2}}}...\right]_{q_i^{\frac{-a_{ij}-2}{2}}}.}
\end{Def}
Note that we have slightly modified the notations and scaling used in \cite{KR}.

\begin{Prop}\label{LuCox}\cite{Lu1, Lu2} The operators $T_i$ satisfy the Weyl group relations:
\Eq{\underbrace{T_iT_jT_i...}_{-a_{ij}'+2} = \underbrace{T_j T_i T_j...}_{-a_{ij}'+2}.\label{coxeter},}
where $-a'_{ij}=\max\{-a_{ij},-a_{ji}\}$.
Furthermore, for $\a_i,\a_j$ simple roots, and an element $w=s_{i_1}...s_{i_k}\in W$ such that $w(\a_i)=\a_j$, we have 
\Eq{T_{i_1}...T_{i_k}(X_i)=X_j} for $X=E,F,K$.
\end{Prop}
\begin{Def} Define the (upper) quantum exponential function as
\Eq{Exp_{q}(x)=\sum_{k=0}^\oo \frac{z^k}{\lceil k\rceil_q!},}
where $\lceil k\rceil_q := \frac{1-q^k}{1-q}$, so that
\Eq{\lceil k\rceil_{q^2}!=[k]_q! q^{\frac{k(k-1)}{2}}.}
\end{Def}
\begin{Thm}\cite{KR, LS} Let $w_0=s_{i_1}...s_{i_N}$ be a reduced expression of the longest element of the Weyl group. Then the universal $R$ matrix is given by
\Eq{\label{RRn}R=\bQ^\half \what{R}(i_N|s_{i_1}...s_{i_N-1})... \what{R}(i_2|s_{i_1})\what{R}(i_1)\bQ^\half,}
where 
\Eq{\bQ:=\prod_{i=1}^n q_i^{H_i\ox W_i},\tab W_i := \sum_{j=1}^n(A\inv)_{ij}H_j,}
\Eq{\label{RRnexp}\what{R}(i):=&Exp_{q_i^{-2}}((q_i-q_i^{-1})E_i\ox F_i),\\
\what{R}(i_l|s_{i_1}...s_{i_{l-1}}):=&(T_{i_1}\inv \ox T_{i_1}\inv)...(T_{i_{l-1}}\inv \ox T_{i_{l-1}}\inv)\what{R}(i_l).}
\end{Thm}

In \cite{Ip4}, the universal $R$ operator is studied in the setting of split real quantum groups $\cU_{q\til{q}}(\g_\R)$, where it is expressed as an element in certain multiplier Hopf algebra, and the Lusztig's isomorphisms $T_i$ are extended to the rescaled positive generators $\be_i$. The rescaled image is useful later to describe the generalized Casimir operators in Section \ref{sec:lowrank:A2} and \ref{sec:lowrank:B2}, but we will not need the explicit construction of the universal $R$ operator in this paper.
\section{The case of $\cU_{q\til{q}}(\sl(2,\R))$}\label{sec:sl2}
The classical Casimir operator $C\in Z(\cU(\sl_2))$ in the case of $\sl_2$ is well-known, and it is given by 
\Eq{C = FE+\left(H+\frac{1}{2}\right)^2.}
In the case of quantum $\cU_q(\sl_2)$, it is also known that the Casimir operator is given by
\Eq{
C &= FE+\left[H+\frac{1}{2}\right]_q^2\\
&= FE + \frac{qK+q\inv K^{-1}}{(q-q\inv)^2}+constant,\nonumber
}
where we formally denote by $K:=q^H$. Rewriting using the rescaling \eqref{smallef}, we denote the rescaled Casimir $\bC$ as
\Eq{\bC &= \bf \be - qK-q\inv K\inv,
}
which can also be rewritten as
\Eq{\bC = \be \bf - q\inv K-qK\inv.}
In the case of split real quantum group, the action of $\bC$ on the positive representation $\cP_\l$ from Proposition \ref{canonicalsl2} can be computed easily:
\begin{Prop}\cite{PT2} $\bC$ acts on $\cP_\l$ as multiplication by the scalar \Eq{C(\l) = e^{2\pi b\l}+e^{-2\pi b\l}.}
\end{Prop}
\begin{proof} Using $e^{2\pi bA}e^{2\pi bB} = q^{2\pi i[A,B]}e^{2\pi b(A+B)}$ We have 
\Eqn{\bC &=\bf \be - qK - q\inv K\inv\\
&=(e^{\pi b(u+\l+2p)}+e^{\pi b(-u-\l+2p)})(e^{\pi b(u-\l-2p)}+e^{\pi b(-u+\l-2p)})-qe^{-2\pi bu}-q\inv e^{2\pi bu}\\
&=q\inv e^{2\pi bu}+e^{2\pi b\l}+e^{-2\pi b\l}+qe^{-2\pi bu}-qe^{-2\pi bu}-q\inv e^{2\pi bu}\\
&=e^{2\pi b\l}+e^{-2\pi b\l}.
}
\end{proof}
We see immediately that the eigenvalue is positive and bounded below by 2 for all values of $\l\in \R$. In particular, since $\cP_\l\simeq \cP_{-\l}$, the eigenvalues of $C(\l)$ and $C(-\l)$ coincides. In this paper, we will see that these results generalize to all higher ranks.

In the case of $\cU_{q\til{q}}(\sl(2,\R))$, the Casimir operator characterizes the positive representations. In particular by studying its spectrum, it provides a decomposition of the tensor products of $\cP_\l$.

\begin{Prop}\cite{NT} The Casimir operator $\bC$ acting on $\cP_\l\ox \cP_\mu$ is unitary equivalent to the self-adjoint operator
\Eq{e^{2\pi bu}+e^{-2\pi bu}+e^{2\pi bp}.\label{kc}}
\end{Prop}
\begin{proof}[Sketch of proof] The coproduct of the Casimir operator is given by
$$\D(\bC)=\bC\ox K+K\inv \ox \bC + K\inv \be\ox \bf K+\bf\ox \be + (q+q\inv)K\inv \ox K.$$
We note that $K,\be$ forms the quantum plane, hence unitary equivalent to a canonical representation given by positive operators. In particular $\be$ is invertible, and one can rewrite the generator $\bf$ as $$\bf = (\bC+qK+q\inv K\inv)\be\inv,$$ which simplifies the expression of $\D(\bC)$. Finally a special function called the \emph{quantum dilogarithm} is used extensively, which provides unitary equivalence between positive  self-adjoint operators of the form $u+v$ and $u$ whenever $uv=q^2vu$ (see e.g. \cite{Ip1} for details). This reduces the 14 terms above into the expression of \eqref{kc}.
\end{proof}

It turns out this operator played the special role of length operator in quantum Teichm\"{u}ller theory, and was studied extensively.
\begin{Lem} \cite{Ip1, Ka}The positive self adjoint operator 
\Eq{e^{2\pi bu}+e^{-2\pi bu}+e^{2\pi bp}\label{kashaev-casimir}} acting on $L^2(\R)$ has a spectral decomposition given by
\Eq{\int_{\R_+}^{\o+}\left(e^{2\pi b\nu}+e^{-2\pi b\nu}\right) d\mu(\nu),}
where the measure $\mu(\nu)$ is given by the quantum dilogarithm.
\end{Lem}
As a Corollary, we have the following theorem, which is the starting point of the research program \cite{FI} of representation theory of split real quantum groups as a tool to construct new classes of braided tensor categories.

\begin{Thm}\cite{PT2} The class of positive representations $\cP_\l$ of $\cU_{q\til{q}}(\sl(2,\R))$ is closed under taking tensor product. We have the following decomposition of tensor products as direct integral
\Eq{\cP_\l \ox \cP_\mu = \int_{\R_+}^\o+ \cP_\nu d\mu(\nu).}
\end{Thm}

\section{Generalized Casimir operators}\label{sec:genCasimir}
The center of $\cU_q(\g)$ (in the sense of Definition \ref{abuse}) is known to be generated by rank $\g$ central elements. There are several ways to construct these generators \cite{Et, RTF, ZGB}. In this paper, we will use a modified version from \cite{KS} and construct the generators by taking certain quantum trace over the fundamental representations $V_k$ of $\cU_q(\g)$. It will be instructive to reproduce the construction here, since we need to switch the order of multiplication later, and slightly modify the quantum trace in order to incorporate the modular double in the positive setting.

\begin{Def} Let $u$ be an invertible element of $\cU_q(\g)$ such that $Ad(u) = S^2$ is the square of the antipode. Let $V$ be a finite dimensional representation of $\cU_q(\g)$. We denote the quantum trace of $x\in \cU_q(\g)$ by
\Eq{Tr|_{V}^q(x):=Tr|_V(xu\inv).}
\end{Def}

\begin{Thm}\label{central}The element 
\Eq{(1\ox Tr|_V^q)(RR_{21})}
belongs to the center of $\cU_q(\g)$, where $R$ is the universal $R$ matrix satisfying the braiding relation \eqref{br}.
\end{Thm}

\begin{proof}  Let
\Eqn{
B &= \{b \in \cU_q(\g)\ox \cU_q(\g): b\D'(a) = \D'(a) b,\forall a\in \cU_q(\g)\},\\
\cI &= \{\l \in \cU_q(\g)^*: \l(ab) = \l(bS^{-2}(a)), \forall a,b\in \cU_q(\g)\}.\\
}

Then $RR_{21} \in B$, and 
$$Tr|_V^q(ab) = Tr|_V(abu\inv) = Tr|_V(bu\inv a)=Tr|_V(bS^{-2}(a)u\inv) = Tr|_V^q(bS^{-2}(a)),$$
hence $Tr|_V^q\in \cI$.
Finally we show that if $b\in B, \l\in \cI$ then $(id\ox \l)(b)$ belongs to the center of $\cU_q(\g)$.

\begin{Lem}
Let $\D(a) = \sum a_{(1)}\ox a_{(2)}$. Then
\Eq{\sum(1\ox S(a_{(1)}))\cdot \D'(a_{(2)}) = a\ox 1 = \sum \D'(a_{(2)}) \cdot (1\ox S\inv(a_{(1)})).}
\end{Lem}
\begin{proof}
using $$(1\ox \D)\D(a) = (\D\ox 1)\D(a) = a_{(1)}\ox a_{(2)}\ox a_{(3)},$$
$$S(a_{(1)})a_{(2)} = \e(a)1,$$
and
$$a_{(2)}\ox \e(a_{(1)})1 = a\ox 1,$$
LHS equals:
\Eqn{&a_{(3)}\ox S(a_{(1)})a_{(2)}=a_{(2)} \ox \e(a_{(1)})1=a\ox 1
}
Similarly for RHS.
\end{proof}
Hence we have
\Eqn{ a\cdot (id\ox \l)(b) &=(id \ox \l)((a\ox 1)\cdot b)\\
&=(id\ox \l)(\sum(1\ox S(a_{(1)}))\cdot \D'(a_{(2)}) \cdot b)\\
\mbox{uses def. of $\cI$}&=(id\ox \l)(\sum \D'(a_{(2)})\cdot b \cdot (1\ox S\inv(a_{(1)}))) \\
\mbox{uses def. of $B$}&=(id\ox \l)(b\cdot \sum \D'(a_{(2)}) \cdot (1\ox S\inv(a_{(1)}))) \\
&=(id\ox \l)(b(a\ox 1))\\
&=(id\ox \l)(b) \cdot a.
}
\end{proof}

Finally, we construct the element $u$, which gives the antipode for both $\cU_q(\g_\R)$ and its modular double counterpart $\cU_{\til{q}}(\g_\R)$. 
\begin{Prop}\label{u} The element $u$ is given by
\Eq{u=K_{2\rho}\til{K}_{2\rho} :=: \prod_i \left(q_i^{2W_i}\right)^{\frac{Q_i}{b_i}},}
where $\rho$ is the sum of fundamental weights and $W_i$ the fundamental coweights given in Definition \ref{rho}.
\end{Prop}
\begin{proof} It suffices to check this on the generators. We check this for $E_j$. It is clear that the $\til{q}$ part commute with $E_j$. Since $W_i = \sum_k (A\inv)_{ik}H_k$, we have
$$[W_i, E_j] =  \sum_k (A\inv)_{ik}[H_k, E_j] =\sum_k (A\inv)_{ik}a_{kj} E_j = \d_{ij}E_j.$$
Hence
\Eqn{uE_j u\inv &=\prod_i\left( q_i^{2W_i}\right)^{\frac{Q_i}{b_i}} E_j \prod_i\left( q_i^{2W_i}\right)^{-\frac{Q_i}{b_i}}=q_j^{\frac{2Q_j}{b_j}}E_j=q_j^2 E_j=S^2(E_j).
}
The proof for $F_j$ is the similar. Taking the power $\frac{1}{b_j^2}$ on both sides proves the case for the modular double counterpart $\til{E_j}, \til{F_j}$:
\Eqn{u\til{E_j} u\inv &=\left(q_j^{\frac{2Q_j}{b_j}}\right)^{\frac{1}{b_j^2}}\til{E_j}=q_j^{\frac{2}{b_j^2}+\frac{2}{b_j^4}}\til{E_j}=\til{q_j}^2 \til{E_j}=S^2(\til{E_j}).
}
\end{proof}

\begin{Thm} The center of $\cU_q(\g)$ is spanned by the $n$ generalized Casimir elements
\Eq{(1\ox Tr|_{V_k}^q)(RR_{21}),}
where $V_k$, $k=1,...,n$ is the $k$-th fundamental representation of $\cU_q(\g)$, and the quantum trace $Tr^q$ is taken with respect to $u$ defined in Proposition \ref{u}.
\end{Thm}
\section{Virtual highest and lowest weights}\label{sec:virtual}
In classical theory, in order to compute the central characters, i.e., the image of the Casimirs on finite dimensional highest weight representation, one writes the Casimir operator in the form of an element in $\cU(\fh)\o+ \cU(\fn^-)\ox\cU(\fn^+)$. By applying the operator to a highest weight vector, only the Cartan part matters, and the eigenvalues can be computed from the weights themselves.

In the case of positive representations, the situation is unfortunately more complicated. The positive representations $\cP_\l\simeq L^2(\R^N)$, $N=\dim l(w_0)$ is infinite dimensional, and there are no highest weight vectors to work with. However, since the representation is irreducible, the Casimir operators still acts as scalars. We have seen in Section \ref{sec:sl2} that $\bC$ acts as a positive scalar in the case of $\cU_{q\til{q}}(\sl(2,\R))$ by computing its action directly. 

In this section, we introduce the notion of \emph{virtual highest/lowest weight vectors} and the corresponding \emph{virtual highest/lowest weights}. From Theorem \ref{FKaction}, we see that $K_i$ acts as multiplication operators on $\cP_\l$. Hence formally its ``eigenvectors" are of the form of generalized functions
\Eq{\d(\vec[x]-\vec[c]):=\prod_{i=1}^N \d(x_i-c_i).} 

One may think of these as distributions on certain subsets $\cW\subset \cP_\l$ of the representation space. One particular choice of $\cW$ is the standard core of our unbounded operators, which is given by the completion of $N$ copies of $\cW\subset L^2(\R)$ where
\Eq{\cW = span \{e^{-\a x^2-\b x}P(x)|\a>0,\b\in\C, P(x)=\mbox{polynomials in $x$}\}.}

Now if we write the actions from Theorem \ref{FKaction} in a normal ordering, i.e. in terms of Laurent polynomials in $U_i$ and $V_i$ where the position operators $U_i=e^{\pi bv_i}$ appears to the right of the momentum (shifting) operators $V_i=e^{\pi b p_i}$, then the (transpose) action on $\d(\vec[x]-\vec[c])$ means substituting $U_i$ by the corresponding numbers. 

\begin{Def}
We call $\vec[c]$ the \emph{point of virtual highest (resp. lowest) weight} if both the actions of $\be_i$ (resp. $\bf_i$) and its modular double counterpart $\til{\be_i}$ (resp. $\til{\bf_i}$) vanishes on $\d(\vec[x]-\vec[c])$ for every $i\in I$.

Let $K_i$ acts as the scalar $e^{2\pi \bi b_i \L_i}$ on $\d(\vec[x]-\vec[c])$. Then $\til{K_i}$ acts as $e^{2\pi \bi b_i\inv \L_i}$. We call $\vec[\L]:=(\L_1,...,\L_n)$ the \emph{virtual highest (resp. lowest) weight} of the positive representation $\cP_\l$ of the modular double $\cU_{q\til{q}}(\g_\R)$.
\end{Def}

We will see that each positive representation $\cP_\l$ has a unique virtual highest/lowest weight. Since we have an explicit formula for the action of $\bf_i$, in this section we will focus on calculating the virtual \emph{lowest} weight vectors.

\begin{Lem}\label{zero} Let $Q_i :=b_i+b_i\inv$. Recall the notations from Definition \ref{usul}. Assume $[u_s+u_l,p_i]=\frac{1}{2\pi \bi}$. Both 
$$[-u_s-u_l]e(2p_i):=e^{\pi (b_su_s+b_lu_l+2b_ip_i)}+e^{\pi (-b_su_s-b_lu_l+2b_ip_i)}$$ and its modular double counterpart (cf. \cite{Ip2,Ip3})
\Eqn{
&\left(e^{\pi (b_su_s+b_lu_l+2b_ip_i)}\right)^{\frac{1}{b_i^2}}+\left(e^{\pi (-b_su_s-b_lu_l+2b_ip_i)}\right)^{\frac{1}{b_i^2}}\\
&=e^{\pi (b_sb_i^{-2} u_s+b_lb_i^{-2} u_l+2b_i\inv p_i)}+e^{\pi (-b_sb_i^{-2} u_s-b_lb_i^{-2} u_l+2b_i\inv p_i)}
}
acts as zero on $\d(b_su_s+b_lu_l+\frac{\bi Q_ib_i}{2})$, and this is a unique solution up to scalar multiples.
\end{Lem}
\begin{proof}
Rewrite the multiplication operator to the right, we have
\Eqn{[-u_s-u_l]e(2p_i):&=e^{\pi (b_su_s+b_lu_l+2b_ip_i)}+e^{\pi (-b_su_s-b_lu_l+2b_ip_i)}\\
&=q_i^\half e^{2\pi b_ip_i}e^{\pi b_su_s+\pi b_lu_l}+q_i^{-\half}e^{2\pi b_ip_i}e^{-\pi b_su_s-\pi b_lu_l}\\
&=e^{2\pi b_ip_i}\left(q_i^\half e^{\pi b_su_s+\pi b_lu_l}+q_i^{-\half}e^{-\pi b_su_s-\pi b_lu_l}\right).
}
Hence the action is zero if and only if
\Eqn{&0=q_i^\half e^{\pi b_su_s+\pi b_lu_l}+q_i^{-\half}e^{-\pi b_su_s-\pi b_lu_l}\\
\iff &q_ie^{2\pi b_su_s+2\pi b_lu_l}=-1\\
\iff &e^{2\pi b_su_s+2\pi b_lu_l+\pi \bi b_i^2} = -1\\
\iff &2\pi b_su_s+2\pi b_lu_l+\pi ib_i^2 = -\pi \bi+2k\pi \bi,\tab k\in\Z\\
\iff &b_su_s+b_lu_l = -\frac{\bi Q_ib_i}{2}+k\bi,\tab k\in\Z.
}
Similarly, the action of the modular double counterpart is zero if and only if
$$b_su_s+b_lu_l = -\frac{\bi Q_ib_i}{2}+k\bi b_i^2,\tab k\in\Z.$$
Hence the only solution that works for both case is when $k=0$, hence $\d(b_su_s+b_lu_l+\frac{\bi Q_ib_i}{2})$ is the unique solution.
\end{proof}

\begin{Cor} In the simply-laced case, $[-u]e(2p)$ and its modular double counterpart act as zero on $\d(u+\frac{\bi Q}{2})$.
\end{Cor}

Next, we need the following result, which describes a general combinatorics of the positive roots. Let $(i_1,...,i_t)$ be any sequence of the root index, and $s_i$ the corresponding reflections.
\begin{Prop}\label{root-com} We have
\Eq{
s_{i_1}s_{i_2}...s_{i_t}(\a_k) = \a_j-\sum_{j=1}^t a_{i_jk}s_{i_1}s_{i_2}...s_{i_{j-1}}(\a_{i_j}),
}
where $\a_k$ is a positive simple root, and $a_{ij}$ the Cartan matrix elements.
\end{Prop}
\begin{proof} We use induction on $t$. The case for $t=0$ is trivial. Assume it holds for $t-1$. Then we have
\Eqn{
&\a_{k}-\sum_{j=1}^t a_{i_jk}s_{i_1}s_{i_2}...s_{i_{j-1}}(\a_{i_j})\\
&=\a_k-\sum_{j=1}^{t-1} a_{i_jk}s_{i_1}s_{i_2}...s_{i_{j-1}}(\a_{i_j}) - a_{i_tk}s_{i_1}...s_{i_{t-1}}(\a_{i_t})\\
&=s_{i_1}s_{i_2}...s_{i_{t-1}}(\a_k)-a_{i_tk}s_{i_1}s_{i_2}...s_{i_{t-1}}(\a_{i_t})\\ 
&=s_{i_1}s_{i_2}...s_{i_{t-1}}(\a_k-a_{i_tk}\a_{i_t})\\ 
&=s_{i_1}s_{i_2}...s_{i_t}(\a_k). 
}
\end{proof}
Recall from \eqref{lb}
$$\vec[\l_b]=\sum \l_ib_i W_i\in \fh_\R ,$$
where $\l=(\l_1,...,\l_n)$ are the parameters of $\cP_\l$, and $W_i$ are the fundamental coweights dual to the positive simple roots (cf. Definition \ref{rho}). We use the variables $v_i$ of the space $\cP_\l$ (cf. Definition \ref{variables}).

\begin{Thm} Fix the positive representation $\cP_\l$ corresponding to the longest element $w_0=s_{i_1}...s_{i_N}$ of the Weyl group. Under the action of Theorem \ref{FKaction}, the point of virtual lowest weight is given by
\Eq{\label{vj}v_j = \frac{1}{b_{i_j}}\left(-\vec[\frac{\bi Q}{2}]-2\vec[\l_b], s_{i_1}...s_{i_{j-1}}(\a_{i_j})\right),}
where 
$$\vec[\frac{\bi Q}{2}]=\sum_i \frac{\bi Q_ib_i}{2}W_i.$$
\end{Thm}
\begin{proof} Recall that the action of $\bf_i$ is given by
$$\bf_i=\sum_{k=1}^m\left[-\sum_{j=1}^{v(i,k)-1} a_{i_j,i}v_j-u_i^k-2\l_i\right]e(2p_{i}^{k}),$$
where $u_i^k$ and $p_i^k$ are the ($k$-th) variables corresponding to the root $i$, $v(i,k)$ is the index of $v$ such that $v_{v(i,k)}=u_i^k$, and $k$ runs over all variables $p_i^k$ corresponding to the root $i$.

For simplicity, let us write $v(i,k)=t$. By Lemma \ref{zero}, it suffices to solve
$$\sum_{j=1}^{t-1} a_{i_ji_t}b_{i_j}v_j+b_{i_t}v_t+2b_{i_t}\l_{i_t}=-\frac{\bi Q_{i_t}b_{i_t}}{2}$$
for all $t$. We will do this by induction on the variables $v_j$.

The case of $v_1$ is easy, since there is only one term:
\Eqn{
b_{i_1}v_1+2b_{i_1}\l_{i_1}&=-\frac{\bi Q_{i_1}b_{i_1}}{2}\\
\iff v_1&=-\frac{\bi Q_{i_1}}{2}-2\l_{i_1}.
}
From Proposition \ref{root-com} substituting $k=i_t$ we have
$$ \a_{i_t}-\sum_{j=1}^{t-1} a_{i_j{i_t}}s_{i_1}s_{i_2}...s_{i_{j-1}}(\a_{i_j})=s_{i_1}s_{i_2}...s_{i_{t-1}}(\a_{i_t}).$$
Assume \eqref{vj} that the variables $v_j$ have been solved for $j<t$. Now take the pairing on the left hand side we get
\Eqn{
&\left(-\vec[\frac{\bi Q}{2}]-2\vec[\l_b],\tab \a_{i_t}-\sum_{j=1}^{t-1} a_{i_j{i_t}}s_{i_1}s_{i_2}...s_{i_{j-1}}(\a_{i_j})\right)\\
&=-\frac{Q_{i_t}b_{i_t}}{2}-2\l_{i_t}b_{i_t}-\sum_{j=1}^{t-1} a_{i_ji_t}b_{i_j}v_j\\
&=b_{i_t}v_t,
}
hence
$$v_t = \frac{1}{b_{i_t}}\left(-\vec[\frac{\bi Q}{2}]-2\vec[\l_b],s_{i_1}s_{i_2}...s_{i_{t-1}}(\a_{i_t})\right)$$
as required.
\end{proof}

\begin{Cor} $K_i$ acts on the point of virtual lowest weight by multiplication by 
\Eq{-q_ie^{2\pi b_i\l_{\s(i)}},}
where $\s$ is the unique involution on the Dynkin diagram corresponding to the action of the longest element $w_0$.
Hence the virtual lowest weight $\vec[\L]=(\L_1,...,\L_n)$ such that $K_i=e^{2\pi \bi b_i \L_i}$ is given by
\Eq{\L_i = \frac{Q_i}{2}-\bi \l_{\s(i)}.}
\end{Cor}
\begin{proof} Again from Proposition \ref{FKaction}, the action of $K_i$ is given by 
\Eqn{K_i &= e\left(-\sum_{k=1}^{N} a_{i_k,i}v_k-2\l_i\right)\\
&:=exp\left(-\pi \sum_{k=1}^N a_{i_k,i}b_{i_k}v_k-2\pi b_i\l_i\right).
}
By Proposition \ref{root-com}, substituting $k=i$ and take $t=N$, we have
$$ \a_i-\sum_{k=1}^{N} a_{i_ki}s_{i_1}s_{i_2}...s_{i_{k-1}}(\a_{i_k})=s_{i_1}s_{i_2}...s_{i_{N}}(\a_i)=-\a_{\s(i)}.$$
Now take the pairing with $-\vec[\frac{\bi Q}{2}]-2\vec[\l_b]$, we get
$$-\frac{\bi Q_ib_i}{2}-2b_i\l_i-\sum_{k=1}^N a_{i_k,i}b_{i_k}v_k = \frac{\bi Q_{\s(i)}b_{\s(i)}}{2}+2b_{\s(i)}\l_{\s(i)}.$$
Also note that $Q_i=Q_{\s(i)}$ and $b_i = b_{\s(i)}$. Hence the action of $K_i$ is given by
$$K_i = \exp(2\pi b_i\l_{\s(i)}+\pi \bi Q_ib_i)=-q_ie^{2\pi b_i\l_{\s(i)}},$$
and the virtual lowest weight $\L_i$ is given by
$$\L_i=\frac{Q_i}{2}-\bi\l_{\s(i)}.$$
\end{proof}
\begin{Rem} Although there is no universal formula for the action of $\be_i$, one can still solve for the point of virtual highest weights case by case from the explicit formula given in \cite{IpTh, Ip3}. In particular the point of virtual highest weight are linear combinations of $\frac{Q_i}{2}$ only. The virtual highest weight is then simply given by 
\Eq{\L_i = -\frac{Q_i}{2}+\bi\l_i.}
\end{Rem}
\section{Central characters}\label{sec:cc}
With the notion of virtual lowest weight, it is now easy to calculate the central characters. Let $V_k$ be the $k$-th fundamental representation of $\cU_q(\g)$. By Theorem \ref{central}, the generalized Casimir operator is given by
$$\bC_k = (1\ox Tr|_{V_k}^q)(RR_{21}) := (1\ox Tr|_{V_k})(RR_{21}(1\ox u\inv)),$$
with $u$ given by Proposition \ref{u}
$$u=\prod_i (q_i^{2W_i})^{\frac{Q_i}{b_i}}= \prod_iq_i^{2W_i}\prod_iq_i^{\frac{2W_i}{b_i^2}}.$$
From the explicit expression of the universal $R$ matrix \eqref{RRn}, it is clear that $\bC_i$ can be written in the form
\Eq{\bC_k = C_{K}+C_{EKF},}
where $C_K\in \cU_q(\fh)$ only depends on $K_i$, and $C_{EKF}\in \cU_q(\fn^+)\ox \cU_q(\fh)\ox \cU_q(\fn^-)$ is linear combinations of terms with $F$ acting on the right. Since the positive representation $\cP_\l$ is irreducible, $\bC_k$ acts as multiplication by a scalar. In particular, the action can be extended to the virtual lowest weight vector. Therefore the terms $C_{EKF}$ do not matter, and all we need to do is to calculate the action of the Cartan part:

\begin{Lem} The eigenvalue $C_k(\l)$ of $\bC_k$ on $\cP_\l$ is given by $(1\ox Tr|_{V_k})(\bQ^2(1\ox u\inv))$ specialized to the actions of $K_i$ on the virtual lowest vector
$$K_i = e^{\pi i}q_ie^{2\pi b_i\l_{\s(i)}}.$$ Recall
$$\bQ = \prod_{i=1}^n q_i^{H_i\ox W_i}, \tab W_i = \sum_{j=1}^n (A\inv)_{ij} H_j.$$
\end{Lem}

Since $\bQ^2(1\ox u\inv)$ is diagonal, the trace is simply given by the sum of the eigenvalues over the weight spaces of $V_k$. In particular the action on the second component is simply substituting its eigenvalues into $W_i$.

Fix a fundamental representation $V_k$, and fix a weight space $\cV$ such that $W_i$ acts as the scalar $\w_i$. Then $\bQ$ acts  on the second component as
$$\bQ= \prod_{i=1}^n K_i^{\w_i},$$
while $u$ acts as
$$u= \prod_iq_i^{2\w_i}\prod_ie^{2\pi \bi\w_i},$$
hence the corresponding term in the trace is given by
$$Tr|_{V_k}(\bQ^2(1\ox u\inv)) = \prod_{i=1}^n (e^{-\pi \bi}q_i\inv K_i)^{2\w_i},$$
and the action on the virtual lowest weight vector $\d(\vec[x]-\vec[c])$ is given by the scalar
$$\prod_{i=1}^ne^{4\pi b_i \l_{\s(i)}\w_i} = e^{4\pi \sum b_i\l_{\s(i)}\w_i}.$$
It is important to note that this is a \emph{positive scalar}, where the dependence on $-q_i$ vanished. Hence combining all the weight spaces, and using the fact that the image is invariant under the Weyl group action on $\l_i$, we can replace $\l_{\s(i)}$ by $-\l_i$ and obtain the main result of the paper:
\begin{Thm}\label{main} Let $\mu_{\cV}$ denote the weight of the weight space $\cV\subset V_k$. The generalized Casimir operators $\bC_k$ acts by the scalar
\Eq{C_k(\l)=\sum_{\cV\subset V_k} \exp\left(-4\pi \mu_\cV(\vec[\l_b])\right),}
where the sum is taken over all the weight spaces of $V_k$, and (cf. \eqref{lb}) $$\vec[\l_b]=\sum_{i=1}^n \l_i b_i W_i\in\fh_\R.$$
\end{Thm}
\begin{Cor} $C_k(\l)$ is positive, and bounded below by $\dim V_k$. 
\end{Cor}

\begin{proof} Since the arguments in the exponents sum to zero by symmetry, by the AM-GM inequality, $C_k(\l)\geq \sum_{\cV\subset V_k} 1 = \dim V_k$ and equality holds if and only if all $\l_i=0$. The dimension $\dim V_k$ of the fundamental representations for each simple Lie type can be found in the appendix.
\end{proof}

We note that if one rescales the parameters $\l_i\mapsto b_i\inv \l_i$, then the eigenvalues $C_k(\l)$ is \emph{independent} on $q$. This is different from the situation in the compact quantum group case, where the Casimirs act as polynomials in $q$.

Finally, we observe that the formula is simply the character of $e^{-4\pi \vec[\l_b]}$ over the fundamental representation $V_k$. Hence the result can be rewritten in terms of the Weyl character formula.

\begin{Cor} We have
\Eq{C_k(\l) = \frac{\sum_{w\in W} sgn(w)e^{-4\pi (w_k+\rho)(w\cdot\vec[\l_b])}}{\prod_{\a\in\D^+}(e^{-2\pi\a(\vec[\l_b])}-e^{2\pi\a(\vec[\l_b])})},
}
where $w_k$ is the $k$-th fundamental weight, which is also the highest weight of $V_k$. In particular, $C_k$ is invariant under the Weyl group action \eqref{lb-weyl} of $\vec[\l_b]$.
\end{Cor}

\begin{Rem} Actually it follows from the proof that for \emph{any} finite dimensional representation $V$, the central element
\Eq{\bC_V=(1\ox Tr|_V^q)(RR_{21})}
acts by \emph{positive} scalar, which is bounded below by $\dim V$. We will call these operators the \emph{positive (generalized) Casimir operators}.
\end{Rem}

The conjecture \cite{Ip5} that the positive representations $\cP_\l$ are closed under taking tensor product then implies
\begin{Con} The coproduct $\D(\bC_k)$ of the generalized Casimir operators $\bC_k$ acting on $\cP_\l\ox\cP_\mu$ is a positive operator, with spectrum bounded below by $\dim V_k$ for all $k=1,..,n$, and they can be simultaneously diagonalized.
\end{Con}
For example, in the case of Type $A_2$, one computes $\D(\bC_1)$ and $\D(\bC_2)$ to have 196 terms each with all coefficients positive, hence they are positive operators acting on $\cP_\l\ox\cP_\mu$. However to diagonalize them seems to be a very difficult task.
 
\section{Discriminant variety}\label{sec:disc}

By the quantum Harish-Chandra homomorphism \cite{KS}, the central characters $C_k(\l)$ on the generalized Casimir operators $\bC_k$ is invariant under the Weyl group action \eqref{lb-weyl} on $\l$, and they take on the same value if and only if $\l$ lies in the same Weyl group orbits. Furthermore, any Laurent polynomials $D(\l)$ spanned by the action of $K_i^{\frac{1}{c}}$ in $\cU_q(\fh)$ (cf. Definition \ref{abuse}) that is invariant under the action of $W$, is given by polynomials in $C_k(\l)$.

We have seen that the generalized Casimir operators $\bC_k$ acts as a positive scalar on the positive representations $\cP_\l$ with parameter $(\l_1,...,\l_n)$ where $\l_i\geq 0$. In particular, we have a map
\Eq{\Phi:\R_{\geq 0}^n &\to \R_{>0}^n\\
(\l_1,...,\l_n) &\mapsto (C_1(\l),...,C_n(\l))\nonumber,
}
which maps the positive quadrant $\R_{\geq 0}^n$ (isomorphic to the positive Weyl chamber) into a simply-connected region $\cR$. The region $\cR$ is obviously bounded by the hypersurfaces given by the image of $\cS_i := \Phi|_{\l_i=0}, i=1,..., n$. Furthermore, these hypersurfaces $\cS_i$ are actually part of a real algebraic variety known as the \emph{discriminant variety} (or \emph{divisor}) defined by polynomials in the space $\R^n$, hence the region $\cR$ is actually a semi-algebraic set \cite{Sai2}. Finally, it appears that the region $\cR$ develops what is called an ADE-type singularity \cite{Sai1} at the point of $\Phi(0,...,0)$, though this requires further clarification from other experts.

The calculations below and the next section are similar in part to \cite{Gr}, which described the discriminant variety in a completely different point of view. Let us consider the following expression
\Eq{D(\l) = \prod_{\a\in\D}(1-e^{-4\pi \a(\vec[\l_b])}),}
where the product is over all roots $\a$ of the root system of $\g$. Then this expression is invariant under the Weyl group action on $\vec[\l_b]$ given by \eqref{lb-weyl}. Hence in particular $D(\l)$ can be expressed in terms of polynomials given by the generalized Casimir $C_k(\l)$. Furthermore, by rescaling $\l_i\to b_i\inv \l_i$, we see that these polynomials are independent on $q$. Finally note that by setting any $\l_i=0$, we have $D(\l) = 0$. Hence
\begin{Prop} $D(\l)$ is a polynomial in $C_k(\l)$. In particular the boundary hypersurfaces of the image $\cR$ of $\Phi$ is part of the algebraic variety defined by $D(\l)=0$ on $\R^n$, and the region $\cR$ is independent on $q$.
\end{Prop}
In the non-simply-laced case, the product can be split over long roots and short roots:
$$D(\l) = \prod_{\a\in \D_{short}}(1-e^{-4\pi \a(\vec[\l_b])})\prod_{\a\in\D_{long}}(1-e^{-4\pi \a(\vec[\l_b])}):=D_sD_l.$$
In particular, $D(\l)$ is reducible, and the boundary is composed of two separate hypersurfaces.

We will consider the examples of the region $\cR$ in type $A_n$ as well as other lower rank case. In all situation, the polynomial $D(\l)$ equals the discriminant of certain polynomials, hence $D(\l)=0$ is also known as the discriminant variety. In geometrical terms, this is the variety describing the critical points of the covering $T\to T/W$ of the split torus.
\subsection{Type $A_n$}\label{sec:disc:An}
In Type $A_n$, the first fundamental representation has weight spaces $\cV_i$ with weight
$$\mu_{\cV_i}(H_j) = \case{1&i=j,\\-1&i=j+1,\\0&\mbox{otherwise,}}$$
and one can find explicitly:
$$C_1(\l) = \sum_{i=1}^n e^{4\pi b(\frac{n-1}{n}\l_1+\frac{n-2}{n}\l_2+...+\frac{1}{n}\l_n)-4\pi b\sum_{k=1}^{i-1}\l_k}:=\sum_{i=1}^n e^{L_i}.$$
Now we take the polynomial
$$P(x) := \prod_{i=1}^n \left(x+e^{L_i}\right).$$
Then it is easy to see that two $L_i$ will coincides if we set any $\l_k=0$. In particular, this means the polynomial $P(x)$ has a double root, hence the discriminant
$\D(P(x)) = 0$
gives the boundary hypersurface for our region $\cR$, and it coincides (up to multiples) with $D(\l)$.
\begin{Rem}$P(x)$ can also be written as
 $$P(x) = x^{n+1}+C_1(\l)x^{n}+...+C_n(\l) x +1,$$
hence the discriminant gives an explicit expression for $D(\l)=0$ in terms of $C_k(\l)$.
\end{Rem}
\section{Examples in low ranks}\label{sec:lowrank}
In this section, we will construct explicitly some generalized Casimir operator, and describe the region $\cR$ of the image of $\Phi$. Recall that the generalized Casimir operator is computed using
$$\bC_k = (1\ox Tr|_{V_k}^q)(RR_{21}).$$
In particular, we know from the discussion in Section \ref{sec:cc} that it is of the form $C_K+C_{EKF}$. Recall our rescaled notation \eqref{smallef}
$$\be_i := \left(\frac{\bi}{q_i-q_i\inv}\right)\inv E_i,\tab \bf_i := \left(\frac{\bi}{q_i-q_i\inv}\right)\inv F_i.$$ 

For simplicity, throughout the calculations let us denote by 
\Eq{k_i := e^{2\pi b_i \l_i},}
i.e. the action of $-q_i\inv K_i$ on the virtual lowest weight vector. Also we will write $X,Y$ and $Z$ as the coordinates taken up by $C_1(\l),C_2(\l)$ and $C_3(\l)$ respectively.
\subsection{Type $A_1$}\label{sec:lowrank:A1}
We have already seen the calculations in Type $A_1$ in Section \ref{sec:sl2}. The Casimir is given by
\Eq{\bC = \bf \be-qK-q\inv K\inv,}
and the eigenvalue is on $\cP_\l$ is given by
\Eq{C(\l) = e^{2\pi b\l} + e^{-2\pi b\l}.}
We have 
\Eq{D(\l) = (1-e^{-4\pi b\l})(1-e^{4\pi b\l}) = 2-e^{-4\pi b\l}-e^{4\pi b\l} = 4-C(\l)^2,}
which gives the obvious boundary $C(\l) = 2$. Note that $-D(\l)$ is the discriminant of the polynomial
\Eq{P(x) = (x+e^{2\pi b\l})(x+e^{-2\pi b\l})=x^2 + C(\l)x+1.}
\subsection{Type $A_2$}\label{sec:lowrank:A2}
The generalized Casimir operator is given by
\Eq{
\bC_1=&K(q^{-2} K_1K_2+K_1\inv K_2 + q^2 K_1\inv K_2\inv -q\inv K_2 \be_1\bf_1-qK_1\inv \be_2\bf_2 + \be_{21}\bf_{12}),\\
\bC_2=&K\inv(q^{2} K_1\inv K_2\inv+K_1 K_2\inv + q^{-2} K_1 K_2 -q K_2\inv \be_1\bf_1-q\inv K_1 \be_2\bf_2 + \be_{12}\bf_{21}),
}
where $K=K_1^{\frac{1}{3}}K_2^{-\frac{1}{3}}$, and
\Eq{\be_{ij}:=\frac{q^{\half}\be_j\be_i - q^{-\half}\be_i\be_j}{q-q\inv},\tab \bf_{ij}:=\frac{q^{\half}\bf_j\bf_i - q^{-\half}\bf_i\bf_j}{q-q\inv}} are the images of the Lusztig's isomorphism extended to positive generators \cite{Ip4}. These expressions of the Casimir operators coincide with the explicit formula given in \cite{RP}.

The central characters are computed to be
\Eq{
C_1(\l) &= e^{\frac{8}{3}\pi b\l_1+\frac{4}{3}\pi b\l_2}+e^{-\frac{4}{3}\pi b\l_1+\frac{4}{3}\pi b\l_2}+e^{-\frac{4}{3}\pi b\l_1-\frac{8}{3}\pi b\l_2},\\
C_2(\l) &= e^{\frac{4}{3}\pi b\l_1+\frac{8}{3}\pi b\l_2}+e^{\frac{4}{3}\pi b\l_1-\frac{4}{3}\pi b\l_2}+e^{-\frac{8}{3}\pi b\l_1-\frac{4}{3}\pi b\l_2}.
}
Taking the discriminant of 
\Eq{P(x) &=(x+e^{\frac{8}{3}\pi b\l_1+\frac{4}{3}\pi b\l_2} )(x+e^{-\frac{4}{3}\pi b\l_1+\frac{4}{3}\pi b\l_2})(x+e^{-\frac{4}{3}\pi b\l_1-\frac{8}{3}\pi b\l_2})\nonumber\\
&=x^3+C_1(\l)x^2+C_2(\l)x+1,} we found the boundary curves to be given by the algebraic equation
\Eq{(XY+9)^2=4(X^3+Y^3+27)}
with a cusp at (3,3) corresponding to the image of $\Phi(0,0)$.

\subsection{Type $A_3$}\label{sec:lowrank:A3}
The full expression of the generalized Casimir operator is too complicated to be listed here. However the central characters are easily computed using the information of the weight spaces of the fundamental representations. We have
\Eq{C_1=&k_1^{\frac{3}{2}}k_2k_3^{\frac{1}{2}}+k_1^{-\frac{1}{2}}k_2k_3^{\frac{1}{2}}+k_1^{-\frac{1}{2}}k_2^{-1}k_3^{\frac{1}{2}}+k_1^{-\frac{1}{2}}k_2^{-1}k_3^{-\frac{3}{2}},\\
C_2=&k_1k_2^2k_3+k_1k_3+k_1^{-1}k_3+k_1k_3^{-1}+k_1^{-1}k_3^{-1}+k_1^{-1}k_2^{-2}k_3^{-1},\\
C_3=&k_1^{\frac{1}{2}}k_2k_3^{\frac{3}{2}}+k_1^{\frac{1}{2}}k_2k_3^{-\frac{1}{2}}+k_1^{\frac{1}{2}}k_2^{-1}k_3^{-\frac{1}{2}}+k_1^{-\frac{3}{2}}k_2^{-1}k_3^{-\frac{1}{2}}.
}
The region is bounded by the discriminant of the polynomial
\Eq{P(x) = x^4+C_1(\l)x^3+C_2(\l)x^2+C_3(\l)x+1,} which is given by
\Eq{
256 - 27 X^4 + 144 X^2 Y - 128 Y^2 - 4 X^2 Y^3 + 16 Y^4 - 192 X Z + 
 18 X^3 Y Z - 80 X Y^2 Z \nonumber\\- 6 X^2 Z^2 
+ 144 Y Z^2 + X^2 Y^2 Z^2 - 
 4 Y^3 Z^2 - 4 X^3 Z^3 + 18 X Y Z^3 - 27 Z^4=0
}
and develops a cusp at (4,6,4).

\subsection{Type $B_2$}\label{sec:lowrank:B2}
The fundamental representations $V_1,V_2$ of $B_2$ has dimension $4$ and $5$ respectively. In order to calculate the generalized Casimir, we choose the following representations. Let $e_{ij}$ be the elementary matrix with $1$ at the $(i,j)$-th position, and 0 otherwise.
\Eqn{
\pi_{V_1}(E_1) &= e_{12}+e_{34},& \pi_{V_1}(E_2)&=e_{23},\\
\pi_{V_1}(F_1) &= e_{21}+e_{43},& \pi_{V_1}(F_2)&=e_{32},\\
\pi_{V_1}(H_1) &= e_{11}-e_{22}+e_{33}-e_{44},& \pi_{V_1}(H_2)&=e_{22}-e_{33},
\\\\
\pi_{V_2}(E_1) &=e_{23}+e_{34},& \pi_{V_2}(E_2)&=e_{12}+e_{45},\\
\pi_{V_2}(F_1) &=[2]_{q^\half}(e_{32}+e_{43}),& \pi_{V_2}(F_2)&=e_{21}+e_{54},\\
\pi_{V_2}(H_1) &= 2e_{22}-2e_{44},& \pi_{V_2}(H_2)&=e_{11}-e_{22}+e_{44}-e_{55},
}
where $[2]_{q^\half}:=q^\half+q^{-\half}$. Recall $q_1=e^{\pi ib_s^2} = q^{\half}, q_2=e^{\pi ib_l^2}=q$. We denote the following rescaled variables for the non-simple root vector, which are all positive self-adjoint operators on $\cP_\l$ \cite{Ip4}.
\Eqn{\be_{21}&:=\frac{q^\half \be_1 \be_2 - q^{-\half}\be_2\be_1}{q-q\inv},&\be_{121}&:=\frac{q^\half \be_2\be_1 - q^{-\half}\be_1\be_2}{q-q\inv},\\
\be_{12}&:=\frac{\be_{21}\be_1-\be_1\be_{21}}{q^\half-q^{-\half}},&\be_{212}&:=\frac{\be_{121}\be_1-\be_1\be_{121}}{q^\half-q^{-\half}},\\
\be_X&:=\frac{q^{\half}\be_{21}\be_1-q^{-\half}\be_1\be_{21}}{q^\half-q^{-\half}},&\be_Y&:=\frac{q^{\half}\be_{121}\be_{21}-q^{-\half}\be_{21}\be_{121}}{q^\half-q^{-\half}},
}
and similarly for $\bf$'s. Then we have

\Eq{
C_1 &= -q^{-2}K_1^2K_2-q\inv K_2+q K_2^{-1}-q^{2}K_1^{-2}K_2^{-1}\\
&+(q^{-3/2}K_1K_2+q^{3/2}K_1^{-1}K_2^{-1})\be_1\bf_1+\be_2\bf_2-q^{-1/2}K_1 \be_{21}\bf_{121}-q^{1/2}K_1\inv \be_{121}\bf_{21}+\be_X\bf_X\nonumber,\\
C_2&=q^{-3} K_1^2K_2^2+q\inv K_1^2+1+q K_1^{-2}+q^{3}K_1^{-2}K_2^{-2}\\
&-[2]_{q^\half}(q\inv K_1+qK_1^{-1})\be_1\bf_1-(q^{-2}K_1^2K_2^2+q^{2}K_1^{-2}K_2^{-1})\be_2\bf_2+\be_1^2\bf_1^2\nonumber\\
&+q^{-2}[2]_{q^\half}K_1 K_2 \be_{121}\bf_{21}+q^2[2]_{q^\half}K_1\inv K_2\inv \be_{21}\bf_{121}-q K_2\inv \be_{212}\bf_{12}-q\inv K_2 \be_{12}\bf_{212}+\be_Y\bf_Y\nonumber,}
with the central characters given by
\Eq{
C_1(\l) &= k_2^2k_2+k_2+k_2\inv+k_1^{-2}k_2\inv,\\
C_2(\l) &= k_1^2k_2^2+k_1^2+k_1^{-2}+k_1^{-2}k_2^{-2}+1.
}
The function $D(\l)$ is reducible as $D(\l) = D_s(\l)D_l(\l)$ where
\Eq{
D_s(\l) &= C_1(\l)^2-4C_2(\l)+4, \\
D_l(\l)&= (C_2(\l)+3)^2-4C_2(\l)^2.
}
In particular the region $\cR$ (which lies in the positive quadrant) is bounded by two curves (cf. Figure \ref{B2}) given by
\Eq{2X=Y+3,\tab X^2=4(Y-1).}
It also develops a cusp-like singularity at the intersection of the two curves at $(4,5)$.
\subsection{Type $B_3$}\label{sec:lowrank:B3}
The fundamental representations has dimension $8,21,7$ respectively, where the second fundamental representation is the 21-dimensional adjoint representation.

The central characters are given by
\Eq{C_1(\l)&=\sum_{\e=\pm}(k_1^3k_2^2k_3)^\e+\sum_{\e=\pm}(k_1 k_2^2 k_3)^\e+\sum_{\e_1,\e_2=\pm} k_1^{\e_1}k_3^{\e_2}, \\
C_2(\l)&=3+\sum_{\a\in \D} k_\a^2\nonumber\\
&=3+k_1^2+k_2^2+k_3^2+k_1^4k_2^2+k_1^2k_2^2+k_2^2k_3^2+k_1^4k_2^2k_3^2+k_1^2k_2^2k_3^2+k_1^4k_2^4k_3^2+\mbox{inverse,}\\
C_3(\l)&=1+\sum_{\e=\pm} k_1^{2\e}+\sum_{\e=\pm}(k_1k_2)^{2\e}+\sum_{\e=\pm} (k_1k_2k_3)^{2\e}.
}

Consider the characteristic polynomial for the third fundamental representation:
\Eq{P(x)&=(x-k_1^2)(x-k_1^{-2})(x-k_1^2k_2^2)(x-k_1^{-2} k_2^{-2})(x-k_1^2k_2^2k_3^2)(x-k_1^{-2} k_2^{-2} k_3^{-2})\nonumber\\
&=1+x^6+(C_3(\l)-1)(x+x^5)+(C_2(\l)-C_3(\l)+1)(x^2+x^4)+(C_1(\l)^2-2C_2(\l)-2)x^3.
}
Then it has multiple roots whenever $\l_i=0$, hence the boundary hypersurfaces are given by the zero level set of the discriminant
$$\D(P(x))=X^2D_sD_l^2=0,$$
where
\Eq{D_s &= X^2-4Y+4Z-8,\\
D_l&=36 + 40 X^2 - 27 X^4 - 132 Y + 90 X^2 Y - 47 Y^2 - 4 Y^3 - 36 Z +\\ 
& 78 X^2 Z - 162 Y Z + 18 X^2 Y Z - 26 Y^2 Z - 27 Z^2 - 6 X^2 Z^2 - \nonumber\\
& 36 Y Z^2 + Y^2 Z^2 + 18 Z^3 - 4 X^2 Z^3 + 6 Y Z^3 + 9 Z^4.\nonumber
}
\subsection{Type $C_3$}\label{sec:lowrank:C3}

The results for type $C_3$ are very similar to type $B_3$. The fundamental representations has dimension $8,21,6$ respectively, where the second fundamental representation is the 21-dimensional adjoint representation.

The central characters are given by
\Eq{
C_1(\l)&=\sum_{\e=\pm}(k_1^3k_2^4k_1^2)^\e+ \sum_{\e=\pm}(k_1k_2^4k_3^2)^\e + \sum_{\e_1,\e_2=\pm} k_1^{\e_1}k_3^{2\e_2},\\
C_2(\l)&=3+\sum_{\a\in \D} k_\a^2\nonumber\\
&=3+k_1^2+k_2^2+k_3^2+k_1^2k_2^2+k_1^2k_2^4+k_1^2k_2^2k_3^2+k_1^2k_2^4k_3^2+k_1^2k_2^4k_3^4+k_2^2k_3^2+\mbox{inverse,}\\
C_3(\l)&=\sum_{\e=\pm}{k_1^\e}+\sum_{\e=\pm}(k_1k_2^2)^\e+\sum_{\e=\pm}(k_1k_2^2k_3^2)^\e.
}

Consider the characteristic polynomials for the third fundamental representation:
\Eq{
P(x)=&(x-k_1^2)(x-k_1^{-2})(x-k_1k_2^2)(x-k_1^{-1}k_2^{-2})(x-k_1k_2^2k_3^2)(x-k_1^{-1}k_2^{-2}k_3^{-2})\nonumber\\
=&1+x^6+C_3(\l)(x+x^5)+(C_3(\l)^2-C_2(\l))(x^2+x^4)+(C_1(\l)+2C_3(\l))x^3.
}
The boundary hypersurface is then given by the zero level set of the discriminant
$$\D(P(x)) = D_l D_s^2=0,$$
where
\Eq{D_l&=X^2 - (2(Z-1)^2-2Y)^2,\\
D_s&=108-27X^2+108Y+36Y^2+4Y^3-54XZ-18XYZ-99Z^2-66YZ^2\\
&-11Y^2Z^2+14XZ^3+30Z^4+10YZ^4-3Z^6.\nonumber
}
We note that the region $\cR$ in this case is extremely narrow compared with the one of type $B_3$ (cf. Figure \ref{B3} and \ref{C3}).
\subsection{Type $D_4$}\label{sec:lowrank:D4}
The fundamental representations of type $D_4$ is special due to the triality symmetries. In particular the central characters for $C_0,C_1$ and $C_3$ (using our labeling, cf. Appendix A) are just permutation of index of each other, while the second fundamental representation corresponding to the center node is the 28-dimensional adjoint representation.
\begin{center}
  \begin{tikzpicture}[scale=.4]
    \draw[xshift=0 cm,thick] (0 cm, 1) circle (.3 cm);
    \draw[xshift=0 cm,thick] (0 cm, -1) circle (.3 cm);
    \foreach \x in {1,...,2}
    \draw[xshift=\x cm,thick] (\x cm,0) circle (.3cm);
   \draw[xshift=0.25 cm] (0 cm,1) -- +(1.4 cm,-1);
   \draw[xshift=0.25 cm] (0 cm,-1) -- +(1.4 cm,1);   
 \foreach \y in {1.15}
    \draw[xshift=\y cm,thick] (\y cm,0) -- +(1.4 cm,0);
    \foreach \z in {2,...,3}
    \node at (2*\z-2,-1) {$\z$};
\node at (-1,-1){$0$};
\node at (-1,1){$1$};
  \end{tikzpicture}
\end{center}

\Eq{
C_0(\l)&=k_0^2k_1k_3k_2^2+k_1k_3k_2^2+k_1k_3+k_1\inv k_3+\mbox{inverse,}\\
C_1(\l)&=k_0k_1^2k_3k_2^2+k_0k_3k_2^2+k_0k_3+k_0\inv k_3+\mbox{inverse,}\\
C_3(\l)&=k_0k_1k_3^2k_2^2+k_0k_1k_2^2+k_0k_1+k_0\inv k_1+\mbox{inverse,}\\
C_2(\l)&=4+\sum_{\a\in\D} k_\a^2\nonumber\\
&=4+k_0^2+k_1^2+k_2^2+k_3^2+(k_0k_2)^2+(k_1k_2)^2+(k_3k_2)^2+(k_0k_1k_2)^2+(k_0k_3k_2)^2+(k_1k_3k_2)^2\nonumber\\
&+(k_0k_1k_3k_2)^2+(k_0k_1k_3k_2^2)^2+\mbox{inverse.}
}

Taking the symmetrized characteristic polynomial of the first (0th node) fundamental representation we get
\Eq{P(x) =&(x+k_0^2k_1k_3k_2^2+k_0^{-2}k_1\inv k_3\inv k_2^{-2})(x+k_1k_3k_2^2+k_1\inv k_3\inv k_2^{-2})\cdot\nonumber\\
&\tab (x+k_1k_3+k_1\inv k_3\inv)(x+k_1\inv k_3+k_1k_3\inv)\nonumber\\
&= x^4+C_0(\l)x^3+(C_2(\l)-4)x^2+(C_1(\l)C_3(\l)-4C_0(\l))x+(C_1(\l)^2+C_3(\l)^2-4C_2(\l)).}
Obviously $P(x)$ has double root whenever $\l_i=0$ (i.e. $k_i=1$). Hence the boundary of the region $\cR$ is given by the zero level set of its discriminant $\D(P(x))$, which in this case is precisely $D(\l)$, a polynomial in $C_i(\l)$ consisting of 88 terms.
\subsection{Type $G_2$}\label{sec:lowrank:G2}
The fundamental representations has dimension 14 and 7 respectively. The explicit expression for the Casimir operators are very complicated in the Chevalley basis so we will omit it here. Using Theorem \ref{main}, we compute the central characters:

\Eq{
C_1(\l) &= 1+\sum_{\e=\pm}(k_1^3k_2^2)^\e+\sum_{\e=\pm}(k_1^3k_2)^\e+\sum_{\e=\pm}k_2^\e+C_2(\l),\\
C_2(\l) &= 1+\sum_{\e=\pm}(k_1^2k_2)^\e+\sum_{\e=\pm}(k_1k_2)^\e+\sum_{\e=\pm}k_1^\e.
}
Note that the first fundamental representation contains a copy of weight spaces equal to that of the second fundamental representations.

We use the symmetrized characteristic polynomial from the second fundamental representations:
\Eq{
P(x) &= (x+k_1+k_1\inv)(x+k_1k_2+k_1\inv k_2\inv)(x+k_1^2k_2+k_1^{-2} k_2\inv)\nonumber\\
&=x^3+(C_2(\l)-1)x^2+(C_1(\l)-2)x+(C_2(\l)-1)^2-2C_1(\l),}
which has double roots whenever $\l_i$=0, hence the boundary curves is given by the zero level set of its discriminant:
$$\D(P(x))=D_sD_l=0,$$
where
\Eq{D_l &= (4X^3-X^2-Y^2-10XY-2X-10Y+7),\\
D_s &= 4(Y+2)-(X+1)^2.
}
The boundary of the region $\cR$ develops a cusp-like singularity at the intersection (14,7) of the two curves (cf. Figure \ref{G2}).
\begin{appendices}
\section{Dimensions of fundamental representations}\label{App:dim}
The detailed calculations of the dimensions below can be found in \cite[Chapter 13]{Car}. Let $d_k = \dim V_k$ be the dimension of the $k$-th fundamental representations. We denote the short roots by black nodes.
\begin{itemize}
\item Type $A_n$: 
\begin{center}
  \begin{tikzpicture}[scale=.4]
    \draw[xshift=0 cm,thick] (0 cm, 0) circle (.3 cm);
    \foreach \x in {1,...,5}
    \draw[xshift=\x cm,thick] (\x cm,0) circle (.3cm);
    \draw[dotted,thick] (8.3 cm,0) -- +(1.4 cm,0);
    \foreach \y in {0.15,...,3.15}
    \draw[xshift=\y cm,thick] (\y cm,0) -- +(1.4 cm,0);
    \foreach \z in {1,...,5}
    \node at (2*\z-2,-1) {$\z$};
\node at (10,-1){$n$};
  \end{tikzpicture}
\end{center}
$$d_k = \binom{n}{k}.$$
\item Type $B_n$: 
\begin{center}
  \begin{tikzpicture}[scale=.4]
    \draw[xshift=0 cm,thick,fill=black] (0 cm, 0) circle (.3 cm);
    \foreach \x in {1,...,5}
    \draw[xshift=\x cm,thick] (\x cm,0) circle (.3cm);
    \draw[dotted,thick] (8.3 cm,0) -- +(1.4 cm,0);
    \foreach \y in {1.15,...,3.15}
    \draw[xshift=\y cm,thick] (\y cm,0) -- +(1.4 cm,0);
    \draw[thick] (0.3 cm, .1 cm) -- +(1.4 cm,0);
    \draw[thick] (0.3 cm, -.1 cm) -- +(1.4 cm,0);
    \foreach \z in {1,...,5}
    \node at (2*\z-2,-1) {$\z$};
\node at (10,-1){$n$};
  \end{tikzpicture}
\end{center}
$$d_k = \case{2^n&k=1,\\\binom{2n+1}{n+k}& k\neq 1.}$$
\item Type $C_n$: 
\begin{center}
  \begin{tikzpicture}[scale=.4]
    \draw[xshift=0 cm,thick] (0 cm, 0) circle (.3 cm);
    \foreach \x in {1,...,5}
    \draw[xshift=\x cm,thick,fill=black] (\x cm,0) circle (.3cm);
    \draw[dotted,thick] (8.3 cm,0) -- +(1.4 cm,0);
    \foreach \y in {1.15,...,3.15}
    \draw[xshift=\y cm,thick] (\y cm,0) -- +(1.4 cm,0);
    \draw[thick] (0.3 cm, .1 cm) -- +(1.4 cm,0);
    \draw[thick] (0.3 cm, -.1 cm) -- +(1.4 cm,0);
    \foreach \z in {1,...,5}
    \node at (2*\z-2,-1) {$\z$};
\node at (10,-1){$n$};
  \end{tikzpicture}
\end{center}
$$d_k = \case{\binom{2n}{n+1-k}-\binom{2n}{n+1+k}& k\neq n,\\ 2n&k=n.}$$
\item Type $D_n$: 
\begin{center}
  \begin{tikzpicture}[scale=.4]
    \draw[xshift=0 cm,thick] (0 cm, 1) circle (.3 cm);
    \draw[xshift=0 cm,thick] (0 cm, -1) circle (.3 cm);
    \foreach \x in {1,...,5}
    \draw[xshift=\x cm,thick] (\x cm,0) circle (.3cm);
    \draw[dotted,thick] (8.3 cm,0) -- +(1.4 cm,0);
   \draw[xshift=0.25 cm] (0 cm,1) -- +(1.4 cm,-1);
   \draw[xshift=0.25 cm] (0 cm,-1) -- +(1.4 cm,1);   
 \foreach \y in {1.15,...,3.15}
    \draw[xshift=\y cm,thick] (\y cm,0) -- +(1.4 cm,0);
    \foreach \z in {2,...,5}
    \node at (2*\z-2,-1) {$\z$};
\node at (10,-1){$n-1$};
\node at (-1,-1){$0$};
\node at (-1,1){$1$};
  \end{tikzpicture}
\end{center}
$$d_k = \case{2^{n-1}&k=0,1,\\\binom{2n}{n+1+k}&i\neq 0,1.}$$
\item Type $E_6$:
\begin{center}
  \begin{tikzpicture}[scale=.4]
    \draw[xshift=0 cm,thick] (0 cm, 0) circle (.3 cm);
    \foreach \x in {1,...,4}
    \draw[xshift=\x cm,thick] (\x cm,0) circle (.3cm);
    \foreach \y in {0.15,...,3.15}
    \draw[xshift=\y cm,thick] (\y cm,0) -- +(1.4 cm,0);
    \foreach \z in {1,...,5}
    \node at (2*\z-2,1) {$\z$};
\draw[xshift=0 cm,thick] (4 cm, -2) circle (.3 cm);
  \draw[xshift=0 cm] (4 cm,-0.25) -- +(0 cm,-1.5);
\node at (4,-3){$0$};
  \end{tikzpicture}
\end{center}
$$(d_0,...,d_5) = (78,27,351,2925,351,27).$$
\item Type $E_7$:
\begin{center}
  \begin{tikzpicture}[scale=.4]
    \draw[xshift=0 cm,thick] (0 cm, 0) circle (.3 cm);
    \foreach \x in {1,...,5}
    \draw[xshift=\x cm,thick] (\x cm,0) circle (.3cm);
    \foreach \y in {0.15,...,4.15}
    \draw[xshift=\y cm,thick] (\y cm,0) -- +(1.4 cm,0);
    \foreach \z in {1,...,6}
    \node at (2*\z-2,1) {$\z$};
\draw[xshift=0 cm,thick] (4 cm, -2) circle (.3 cm);
  \draw[xshift=0 cm] (4 cm,-0.25) -- +(0 cm,-1.5);
\node at (4,-3){$0$};
  \end{tikzpicture}
\end{center}
$$(d_0,...,d_6) = (912,133,8645,365750,27664,1539,56).$$
\item Type $E_8$:
\begin{center}
  \begin{tikzpicture}[scale=.4]
    \draw[xshift=0 cm,thick] (0 cm, 0) circle (.3 cm);
    \foreach \x in {1,...,6}
    \draw[xshift=\x cm,thick] (\x cm,0) circle (.3cm);
    \foreach \y in {0.15,...,5.15}
    \draw[xshift=\y cm,thick] (\y cm,0) -- +(1.4 cm,0);
    \foreach \z in {1,...,7}
    \node at (2*\z-2,1) {$\z$};
\draw[xshift=0 cm,thick] (4 cm, -2) circle (.3 cm);
  \draw[xshift=0 cm] (4 cm,-0.25) -- +(0 cm,-1.5);
\node at (4,-3){$0$};
  \end{tikzpicture}
\end{center}
$$(d_0,...,d_7) = (147250,3875,6696000,6899079264,146325270,2450240,30380,248).$$
\item Type $F_4$:
 \begin{center}
  \begin{tikzpicture}[scale=.4]
    \draw[thick] (-2 cm ,0) circle (.3 cm);
	\node at (-2,-1) {$1$};
    \draw[thick] (0 ,0) circle (.3 cm);
	\node at (0,-1) {$2$};
    \draw[thick,fill=black] (2 cm,0) circle (.3 cm);
	\node at (2,-1) {$3$};
    \draw[thick,fill=black] (4 cm,0) circle (.3 cm);
	\node at (4,-1) {$4$};
    \draw[thick] (15: 3mm) -- +(1.5 cm, 0);
    \draw[xshift=-2 cm,thick] (0: 3 mm) -- +(1.4 cm, 0);
    \draw[thick] (-15: 3 mm) -- +(1.5 cm, 0);
    \draw[xshift=2 cm,thick] (0: 3 mm) -- +(1.4 cm, 0);
  \end{tikzpicture}
\end{center}
$$(d_1,d_2,d_3,d_4)=(52,1274,273,26).$$
\item Type $G_2$: 
\begin{center}
  \begin{tikzpicture}[scale=.4]
    \draw[thick] (0 ,0) circle (.3 cm);
	\node at (0,-1) {$1$};
    \draw[thick,fill=black] (2 cm,0) circle (.3 cm);
	\node at (2,-1) {$2$};
    \draw[thick] (30: 3mm) -- +(1.5 cm, 0);
    \draw[thick] (0: 3 mm) -- +(1.5 cm, 0);
    \draw[thick] (-30: 3 mm) -- +(1.5 cm, 0);
  \end{tikzpicture}
\end{center}
$$(d_1,d_2) = (14,7).$$
\end{itemize}

\section{Boundary regions of central characters}\label{App:graph}
We list here the graphs of the boundary region $\cR$ of the image of $\Phi$ of the central characters in the rank 2 and 3 cases. In all cases we note that the region develops a cusp like singularity at the point $$\Phi(0,...,0) = (d_1,...,d_n),$$ where $d_k$ is the dimension of the fundamental representations. The grid lines in the graphs correspond to the parameters $\l_i\in[0,\infty)$.

\newpage
\begin{figure}[h!]
\centering
\begin{subfigure}{.5\textwidth}
\centering
\includegraphics[width=50mm]{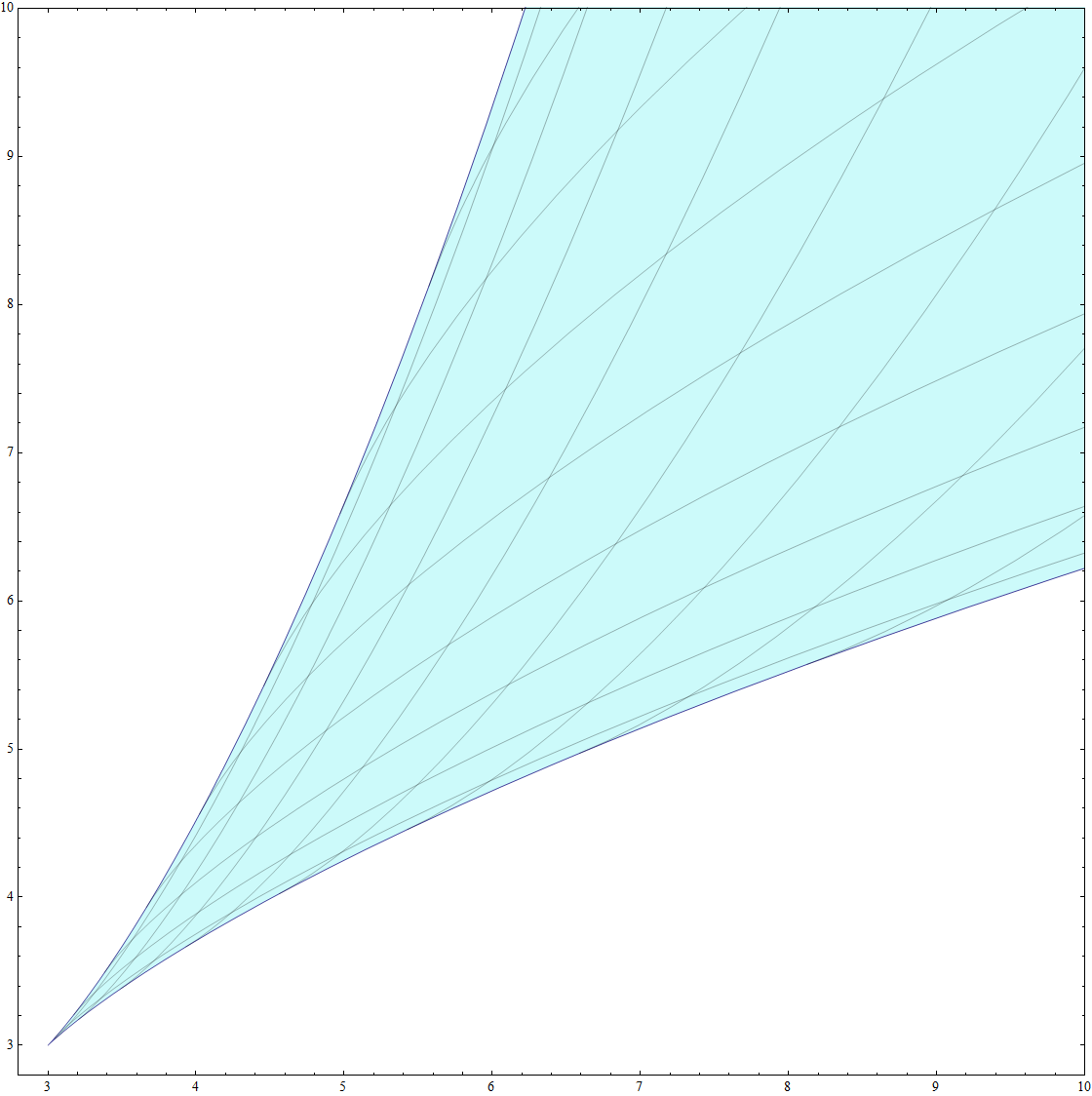}
\caption{Type $A_2$}
\label{A2}
\end{subfigure}%
\begin{subfigure}{.5\textwidth}
\centering
\includegraphics[width=50mm]{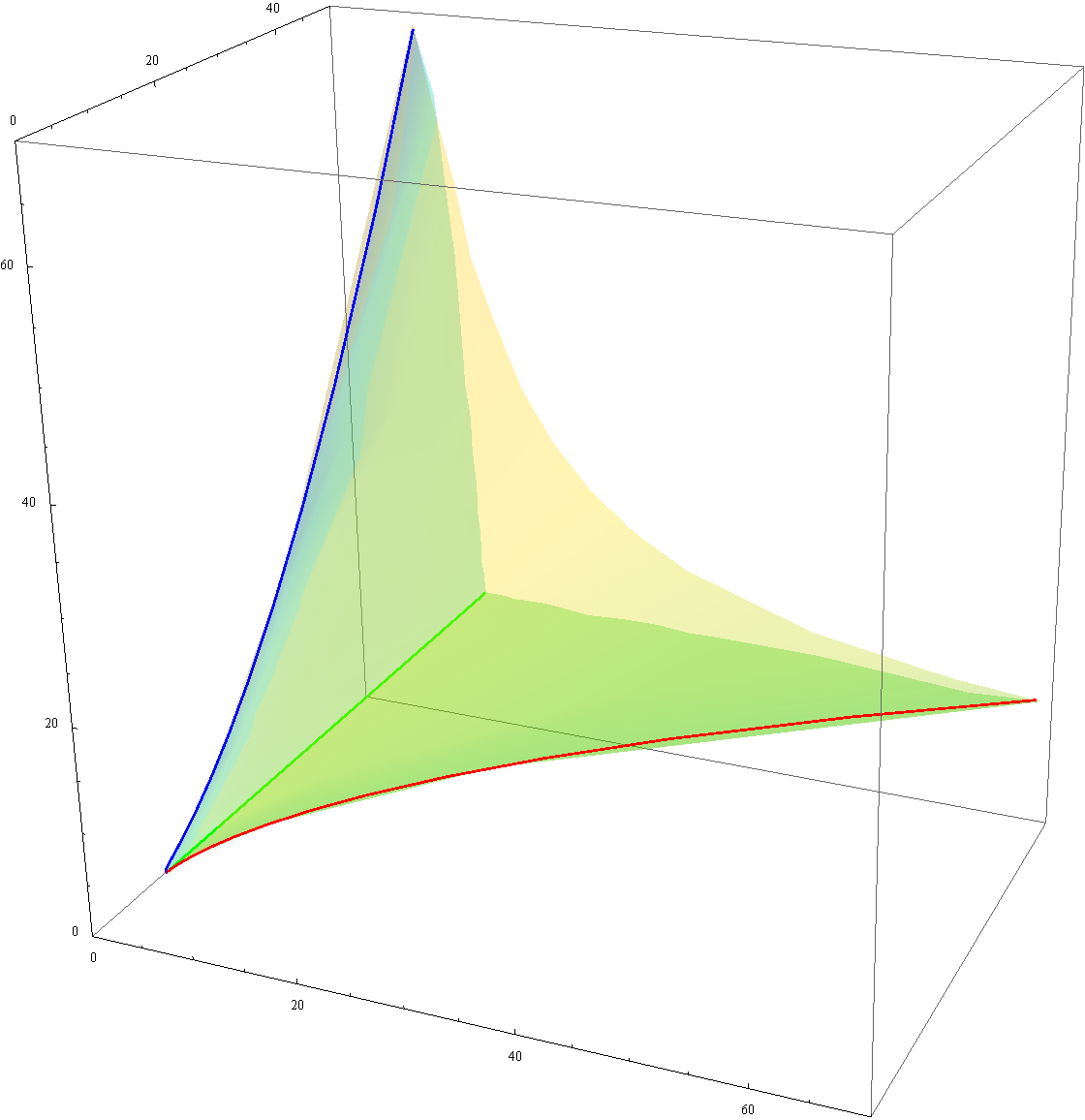}
\caption{Type $A_3$}
\label{A3}
\end{subfigure}
\begin{subfigure}{.5\textwidth}
\centering
\includegraphics[width=50mm]{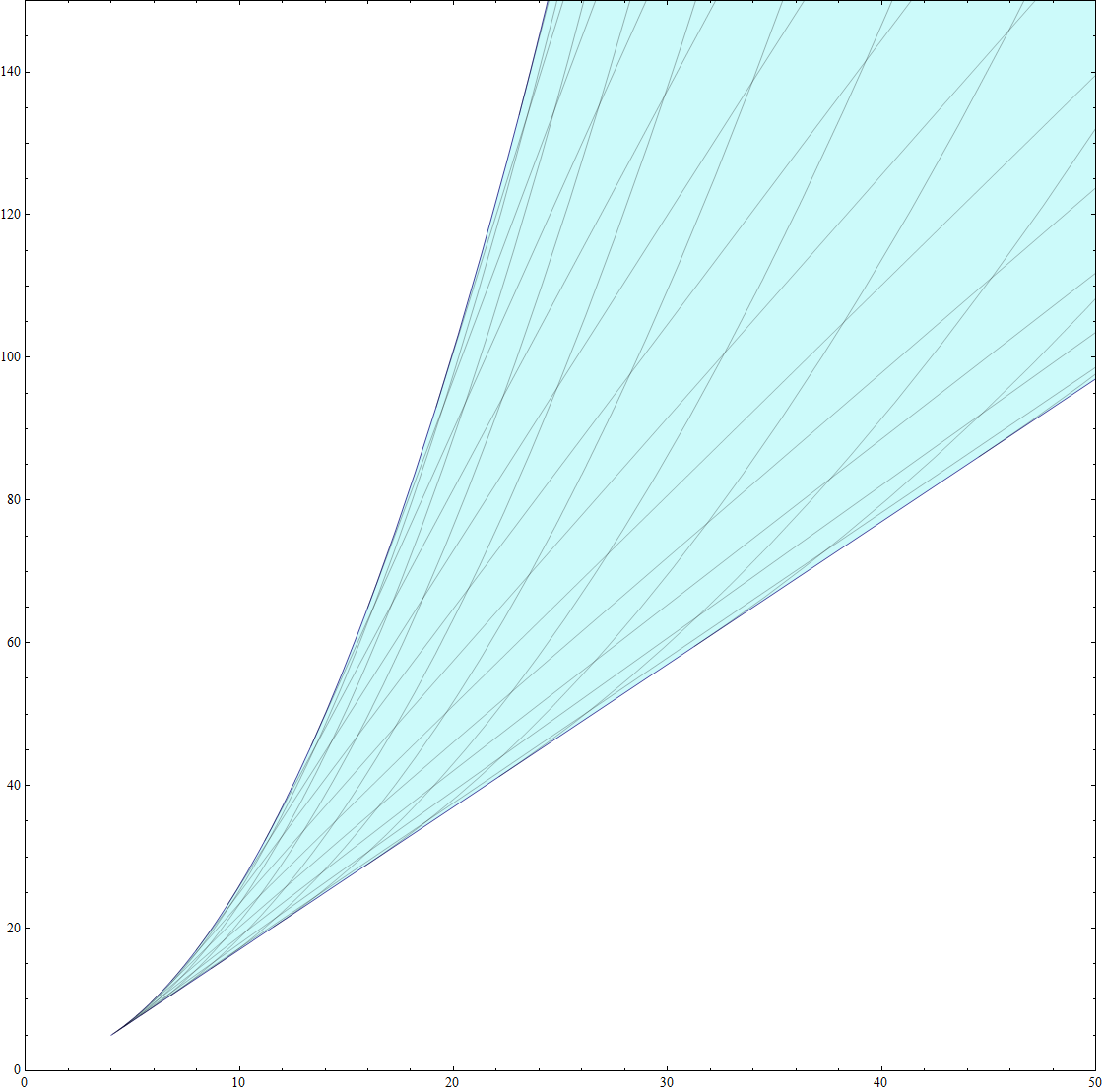}
\caption{ Type $B_2$}
\label{B2}
\end{subfigure}%
\begin{subfigure}{.5\textwidth}
\centering
\includegraphics[width=50mm]{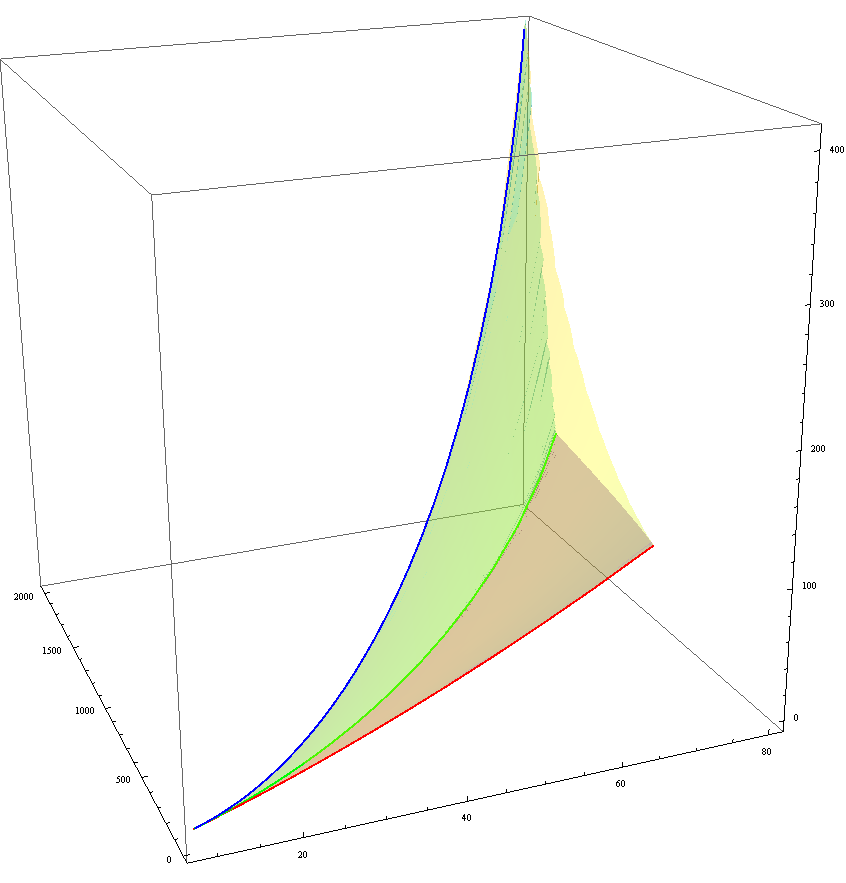}
\caption{Type $B_3$}
\label{B3}
\end{subfigure}
\begin{subfigure}{.5\textwidth}
\centering
\includegraphics[width=50mm]{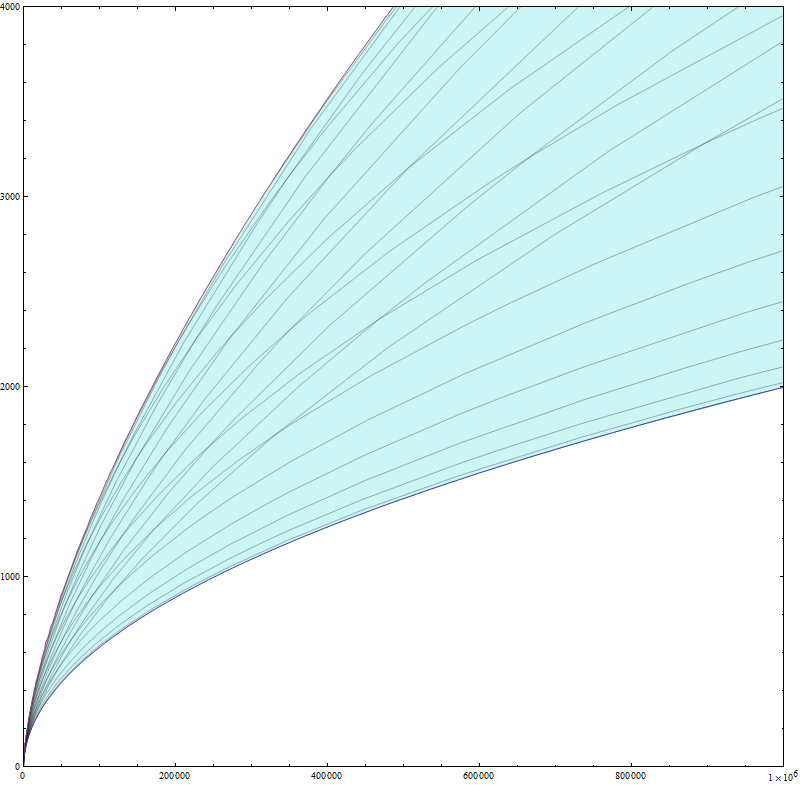}
\caption{Type $G_2$}
\label{G2}
\end{subfigure}%
\begin{subfigure}{.5\textwidth}
\centering
\includegraphics[width=50mm]{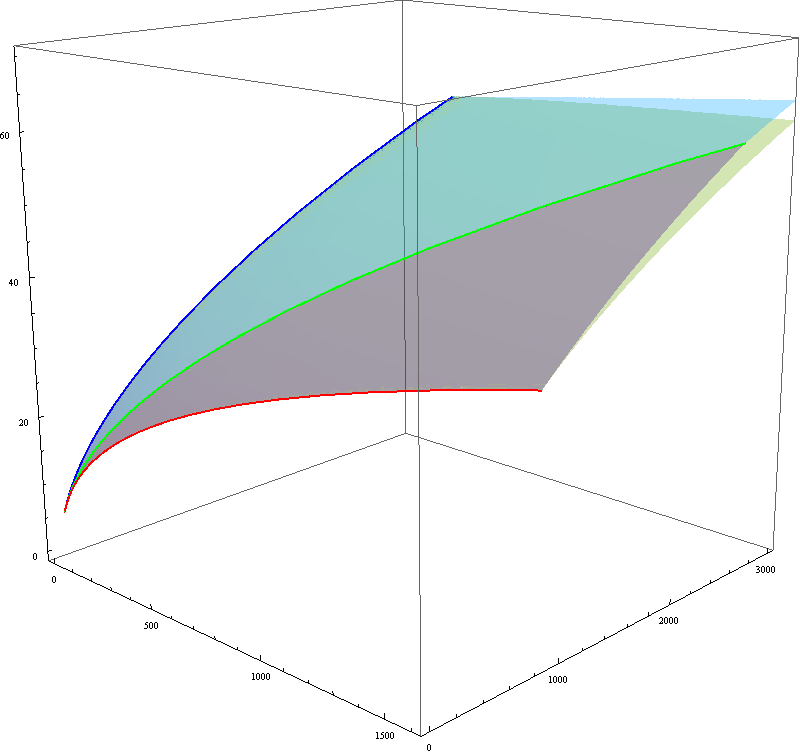}
\caption{Type $C_3$}
\label{C3}
\end{subfigure}

\caption{The boundary region $\cR:=\Phi(\R_{\geq0}^n)$ in rank $2$ and $3$.}
\end{figure}

\end{appendices}

\end{document}